\pdfoutput=1
\RequirePackage{ifpdf}
\ifpdf 
\documentclass[pdftex]{sigma}
\else
\documentclass{sigma}
\fi

\usepackage[all]{xy}

\let\ve\mathbf
\def\rmref#1{{\rm\ref{#1}}}
\def\I#1#2{\mathrm{I}_{#1#2}}
\def\II#1#2{\mathrm{II}_{#1#2}}
\def\barI#1#2{\bar{\mathrm{I}}_{#1#2}}
\def\barII#1#2{\bar{\mathrm{II}}_{#1#2}}
\def\tilI#1#2{\widetilde{\mathrm{I}}_{#1#2}}
\def\tilII#1#2{\widetilde{\mathrm{II}}_{#1#2}}
\def\h#1#2{h_{#1#2}}
\def\curv#1{{\mathop{\rm c}}_{#1}}
\def\nc#1{{\mathop{\rm nc}}_{#1}}
\def\gc#1{{\mathop{\rm gc}}_{#1}}
\def\gtor#1{{{\rm gt}}_{#1}}

\def\omegaIII{\omega_{\hbox{\sc iii}}}

\def\detI{{\rm det\,I}}
\def\detII{{\rm det\,II}}

\def\f{\omega}
\let\epsilon\varepsilon
\def\m{\mu}
\def\k{\kappa}
\def\l{\lambda}

\def\1{{-1}}

\def\omegaa{\varphi^+}
\def\omegab{\varphi^-}

\def\E{\mathrm{E}}
\def\F{\mathrm{F}}
\def\am{\mathop{\rm am}\nolimits}
\def\sn{\mathop{\rm sn}\nolimits}
\def\cn{\mathop{\rm cn}\nolimits}
\def\dn{\mathop{\rm dn}\nolimits}

\def\arccn{\mathop{\rm arccn}\nolimits}

\def\sign{\mathop{\rm sign}\nolimits}

\def\zb#1{\hbox to 0pt{\hss#1\hss}}
\def\hm{\hphantom{-}}
\def\eqref#1{{\rm(\ref{#1})}}

\newtheorem{Theorem}{Theorem}[section]
\newtheorem{Corollary}[Theorem]{Corollary}
\newtheorem{Proposition}[Theorem]{Proposition}
 { \theoremstyle{definition}
\newtheorem{construction}[Theorem]{Construction}
\newtheorem{Definition}[Theorem]{Definition}
\newtheorem{Example}[Theorem]{Example}
\newtheorem{Remark}[Theorem]{Remark} }

\numberwithin{equation}{section}

\def\sy#1{\expandafter\csname #1\endcsname}
\def\newsy#1{\expandafter\def\csname #1\endcsname}

\def\dif{\mathop{}\!\mathrm d}
\def\pd#1#2{\frac{\partial#1}{\partial#2}}

\newsy{h11}{\h11}
\newsy{h12}{\h12}
\newsy{h22}{\h22}

\begin{document}

\newcommand{\arXivNumber}{2403.12626}

\renewcommand{\PaperNumber}{029}

\FirstPageHeading

\ArticleName{On Integrable Nets in General\\ and Concordant Chebyshev Nets in Particular}
\ShortArticleName{On Integrable Nets in General and Concordant Chebyshev Nets in Particular}

\Author{Michal MARVAN}

\AuthorNameForHeading{M.~Marvan}

\Address{Mathematical Institute in Opava, Silesian University in Opava,\\ Na Rybn\'{\i}\v{c}ku 1, 746\,01~Opava, Czech Republic}
\Email{\href{mailto:michal.marvan@math.slu.cz}{michal.marvan@math.slu.cz}}

\ArticleDates{Received March 20, 2024, in final form March 31, 2025; Published online April 28, 2025}

\Abstract{We consider general integrable curve nets in Euclidean space as a particular integrable geometry invariant with respect to rigid motions and net-preserving reparameterisations. For the purpose of their description, we first give an overview of the most important second-order invariants and relations among them. As a particular integrable example, we reinterpret the result of I.S.~Krasil'shchik and M.~Marvan (see Section~2, Case 2 in [{\it Acta Appl. Math.}~\textbf{56} (1999), 217--230]) as a curve net satisfying an $\mathbb R$-linear relation between the Schief curvature of the net and the Gauss curvature of the supporting surface. In the special case when the curvatures are proportional (concordant nets), we find a correspondence to pairs of pseudospherical surfaces of equal negative constant Gaussian curvatures. Conversely, we~also show that two generic pseudospherical surfaces of equal negative constant Gaussian curvatures induce a concordant Chebyshev net. The construction generalises the well-known correspondence between pairs of curves and translation surfaces.}

\Keywords{integrable surface; integrable curve net; differential invariant; pseudospherical surface; Chebyshev net; concordant net}

\Classification{37K10; 53A05; 53A55; 53A60}

\vspace{-2mm}

\section{Introduction}

Classical integrable geometry includes integrable classes of surfaces in Euclidean
space as the most familiar instance~\cite{Bob,G-T,R-S,Sym}.
Integrability is mostly understood in the sense of soliton theory.
Numerous examples are known, often originating in the nineteenth century.
A~handful have been characterised in terms of differential invariants of surfaces.
In particular, Bianchi \mbox{\cite[Section~99]{LB-I}} characterised the isometry classes of surfaces
of revolution (which correspond one-to-one to planar curves).
Well-known are also surfaces satisfying
$\Delta(1/H) = 0$, Bianchi surfaces, and some others~\cite{Bob,G-T}.
A~number of known integrable geometries have been characterised in terms
of curve invariants. These include, for instance, the Hasimoto surfaces swept by
curves moving according to geometrically determined dynamics~\cite{HH-1972,R-S}
or the Razzaboni surfaces formed by nets of Bertrand curves~\cite{WKS-2003}.

However, quite rare have been successful classification attempts.
Those known to the author are limited to integrable Weingarten surfaces and their
evolutes, see~\cite{B-M-2010} and references therein, which revealed nothing
unrelated to nineteenth-century geometry.

We consider integrable nets as integrable geometries characterisable in
terms of net invariants.
The paper has grown out of our earlier result~\cite[Section~2]{K-M} on  integrable
Gauss--Mainardi--Codazzi systems under Chebyshev parameterisation.
 Two main unsolved problems~were:\looseness=-1
\begin{enumerate}\itemsep=0pt
  \item[(A)] finding the geometric meaning of the result and
  \item[(B)] constructing explicit solutions.
\end{enumerate}
Sections~\ref{sect:nets} to~\ref{sect:int:Cheb} and Appendix~\ref{app:rel}
pertain to problem (A) and Sections~\ref{sect:int:Cheb} to~\ref{sect:ex} to problem~(B).

Section~\ref{sect:nets} briefly reviews nets,
emphasising their description as direction pairs.
Section~\ref{sect:inv} reviews second-order differential invariants,
including  the Schief curvature \cite[Section~3.1]{WKS-2007}.
Section~\ref{sect:Chebyshev} reviews general Chebyshev nets
and characterises them in terms of two scalar invariants.
Section~\ref{sect:integrability} introduces integrable classes of nets in analogy with
integrable classes of surfaces and explains their main differences.
Relations among invariants are relegated to Appendix~\ref{app:rel}.
In Section~\ref{sect:int:Cheb}, we turn to integrable Chebyshev
parameterisations found in~\cite[Section~2]{K-M} and
easily recognise them as classes of nets, which answers problem (A).

The first part may seem unnecessarily extensive,
compared to the simple answer it eventually gives to problem (A).
However, this part has also the concurrent goal of compensating for the lack of
suitable survey literature on nets and their invariants, opening the way to more
classification results related to integrable geometries, and possibly also to a new
interpretation of old results in planned follow-ups to this article.

As for problem (B), paper~\cite{K-M} only provided a zero-curvature representation
(ZCR), which is a standard starting point for obtaining exact solutions~\cite{BAD-1981,VBM-2008}.
However, we have not been able to turn the ZCR into solutions.

In this paper, we manage to solve problem (B) in the case of {\it concordant\/}
Chebyshev nets, characterised by the proportionality of the Gauss and Schief curvatures.
For this class, vector conservation laws are obtained in Section~\ref{sect:vcl}.
With their help, we establish a correspondence between concordant Chebyshev nets
and pairs of pseudospherical surfaces of equal curvatures,
providing a geometric solution to problem (B).
The passage from concordant nets to pairs of pseudospherical surfaces is covered in
Section~\ref{sect:ccnets->PS}, the opposite direction in Section~\ref{sect:PS->ccnets}.
The~construction generalises the well-known correspondence between
translation surfaces and pairs of
curves~\cite{LPE-1909,F-T-2007,K-T-2021,SL-1878}
and
provides a more or less straightforward way to obtain examples of exact concordant
Chebyshev nets, see Section~\ref{sect:ex}.

For simplicity, our exposition is local;
smoothness is assumed everywhere.

\section{Nets}
\label{sect:nets}

We consider nets  immersed in the Euclidean space
$\ve E^3$.
They are a classical object of interest
in differential geometry~\cite{LB-I,LB-II,GD-I,GD-II,GD-III} and
have numerous applications, especially in construction and
architecture~\cite{ENK,M-D-T-F-B,P-E-V-W,XT-2022}.
Examples include the
asymptotic, characteristic, Chebyshev, circular, cone-, conformal, conjugate,
equal path, equiareal, geodesic, Hasimoto, LGT, Liouville, orthogonal, principal,
Razzaboni, Voss--Guichard, wobbly nets,
and plenty of others (e.g.,~\cite{GD-I,YaSD,LPE-1904,BG,GMG,K-M-T,RK-1979,RK-1982,RL-RC,M-D-T-F-B,R-S,RS-1933,RS-1970,V-B-2010,W-P-2022}
and references therein).
Nets also appear as substructures of richer structures such as $n$-webs,
see~\cite{SIA-2021} and references therein.
Still other nets appear as smooth limits of discrete nets,
which are obligatory substructures of discrete surfaces~\cite{AB-UP,B-P-R,B-S}.

By a {\it local parameterisation} or simply a {\it parameterisation} of a surface
$S \subset \ve E^3$ we mean a~diffeomorphism $\ve r \colon  U \to V$,
where $U \subseteq \mathbb R^2$ is an open subset of the
{\it parameter space} $\mathbb R^2$ and
$\ve r U = V \subseteq S$
is an open subset of the surface $S$.
In this paper, $\ve r$ and $S$ are always related in this way.
\looseness-1

Viewed as maps $\ve r \colon  U \to \mathbb R^3$, parameterisations can be added and
multiplied by functions $U \to \mathbb R$.
Thus, parameterisations $\ve r \colon  U \to \mathbb R^3$ form a $C^\infty U$-module.

A net on a surface $S$ can be introduced in various equivalent ways,
in particular as a pair of transversal foliations of $S$ by curves or as a pair of
transversal direction fields on $S$.
Both exist in oriented and non-oriented versions.

\begin{Definition} \label{def:net}
A {\it foliation} of an open set $V \subseteq S$ is the partition of $V$ into the
level sets ${f = {\rm const}}$ of a function $f \colon  V \to \mathbb R$, $\dif f \ne 0$.
Foliations $f_1 = {\rm const}$ and $f_2 = {\rm const}$ are {\it transversal}
if $\dif f_1 \wedge \dif f_2 \ne 0$.
Locally, a {\it net\/} on a surface $S$ is a transversal pair of foliations.
If $\dif f_1 \wedge \dif f_2 = 0$ at isolated points or lines, these are referred to as
{\it singular}.
\end{Definition}

The surface $S$ is said to be {\it supported} by the net.

\begin{Definition} \label{def:net:isoparam}
In the notation of Definition~\ref{def:net},
let $\ve r$ be a parameterisation of $S$.
Then functions $x_1 = f_1 \circ\ve r$, $x_2 = f_2 \circ\ve r$  are called the
{\it family parameters}, with respect to which
the curves $f_i =$ const are the {\it isoparametric} curves.
The net is denoted by $\ve r(x_1,x_2)$ and said to be {\it isoparametric}.
\end{Definition}

Every net on a surface $S$ is locally isoparametric if we choose $x_1$, $x_2$ from
Definition~\ref{def:net:isoparam} as the local parameters.

Obviously, regular local reparameterisations
\begin{align}\label{eq:net:transf}
x_1' = x_1'(x_1), \qquad x_2' = x_2'(x_2)
\end{align}
preserve the curve families.
Locally, nets can be identified with the equivalence classes of
parameterisations modulo reparameterisations~\eqref{eq:net:transf}.%
\footnote{In the literature, transformations~\eqref{eq:net:transf} are sometimes
called {\it Sannian transformations}~\cite{YaSD, GS}.}

Differential invariants of curve nets can depend on the orientation.
{\it Oriented nets} can be introduced as the equivalence classes of
parameterisations modulo reparameterisations~\eqref{eq:net:transf} satisfying
${\rm d}x_i'/{\rm d}x_i > 0$.

Working with parameterisations is not entirely convenient when dealing with several
different nets on a surface simultaneously.
This can be remedied by employing direction pairs, oriented or non-oriented.
For counterparts used in computer graphics see~\cite[Section~2]{Vax-2016}.

\begin{Definition}
A {\it direction field} $[X]$
represented by a nowhere vanishing vector field $X$ on an open set
$V \subseteq S$ is defined by
\[
[X]
 = \{ f X \mid f \in C^\infty S, \, f \lessgtr 0 \}.
\]
In the oriented version,
\[
[X]
 = \{ f X \mid f \in C^\infty S, \, f > 0 \}.
\]
A {\it direction pair}
is an ordered pair $([X_1], [X_2])$ of two distinct direction fields.
\end{Definition}

The fields can be specified in the parameter domain
$U \subseteq \mathbb R^2$ and mapped to $S$ by the tangent mapping
$\ve r_* \colon  TU \to TS$, which is tacitly understood in this paper.

Needless to say, nets and direction pairs mutually correspond.
In the non-oriented setting,
tangent vector fields to curves of a net represent a direction pair,
while the trajectories of the generating vector fields form a net.
Let us, however, remark that a direction pair can exist globally even if the
corresponding net of trajectories does not (recall the irrational flow on a~torus).\looseness=-1

Obviously, transformations
\begin{gather}
\label{eq:dp:transf}
X_i' = f_i X_i,
\end{gather}
where $f_i \in C^\infty S$, $f_i \lessgtr 0$, preserve non-oriented direction fields.
Oriented direction fields are preserved if functions $f_i$ are positive.

Transformations~\eqref{eq:net:transf}  and~\eqref{eq:dp:transf} mutually correspond.
In the non-oriented setting, a direction field $[X]$ in $\mathbb R^2$
corresponds to a vector field $X$ modulo the equivalence~$X \equiv f X$,
$f \lessgtr 0$,
which corresponds to a linear homogeneous first-order PDE,
which has a general solution of the form~$F(x)$, where $X x = 0$ and
$\dif F \ne 0$.
In the oriented setting, the gradients $\operatorname{grad} f_i$ are naturally oriented and have to
correspond to the orientations of $X_i$ and of the surface $S$, which must be orientable.

\begin{Remark}
Weise~\cite[Section~1]{KHW-1940} approached nets as isotropic directions
of a conformal class of Lorentzian metrics,
which became a common approach in the former Soviet literature~\cite{VFK-II,APN,VIS}.
This approach provides a connection to binary differential equations~\cite{B-T-1995},
but does not distinguish between pairs $([X_1], [X_2])$ and $([X_2], [X_1])$,
which prohibits asymmetrically defined nets.
\end{Remark}

\begin{Definition}
Vector fields $X_1$, $X_2$ are said to be the
{\it commuting representatives} of a direction pair $([X_1],[X_2])$
if they commute.
\end{Definition}

\begin{Proposition}
Every direction pair locally possesses commuting representatives.
\end{Proposition}

\begin{proof}
These can be obtained as the vector fields
$\partial/\partial x_1$ and $\partial/\partial x_2$ for
the family parameters~$x_1$,~$x_2$
(see Definition~\ref{def:net:isoparam})
of the corresponding net.
\end{proof}

\begin{Definition}
Denoting by ${\rm I}$ the metric of $S$, ${\rm I}(X,Y) = X \ve r \cdot Y \ve r$, the
{\it unit representative} $\widehat{X}$ of a direction field $[X]$ is defined by
\[
\widehat{X} = \frac{X}{\sqrt{{\rm I}(X, X)}},
\]
choosing the positive square root.
\end{Definition}

Obviously, ${\rm I}\bigl(\widehat{X}, \widehat{X}\bigr) = 1$, while the trajectories of $\widehat{X}$
are naturally parameterised by the arc length.

Thus, every net has commutative representatives and unit representatives,
which are normally different.
The coincidence of these representatives characterises Chebyshev nets, see
Proposition~\ref{prop:Chebyshev}\,(iii).

In what follows, we shall need some descriptors adopted from surface theory.
Firstly,
\[
\ve n
 = \frac{X_i \ve r \times X_j \ve r}{\sqrt{\|X_i \ve r \times X_j \ve r\|}}
\]
are the unit normal vector to the supported surface.
Secondly, the {\it fundamental coefficients} are defined by
\begin{gather}\label{eq:I-II}
\I ij = X_i \ve r \cdot X_j \ve r, \qquad
\II ij = X_j X_i \ve r \cdot \ve n.
\end{gather}
These are analogues of the coefficients of the fundamental forms
and coincide with them when $X_i = \partial/\partial x_i$ are the coordinate fields.

The expressions $\I ij$, $\II ij$ are symmetric in $i$, $j$ and invariant with
respect to rigid motions.
The symmetry of $\II ij$ is obvious from
$[X_i, X_j]\ve r \cdot \ve n = 0$.

\section{Invariants of nets}\label{sect:inv}

Invariants of nets have been pioneered by Aoust~\cite{AA} and Weise~\cite{KHW-1940}.
For an overview in various settings,
see~\cite{YaSD,APN,WS-I,WS-II,WS-III,VIS}.
For differential invariants in general, see~\cite{A-L-V}.
Here we recall useful first- and
second-order scalar differential invariants in terms of direction pairs.
In fact, only five of the invariants, namely $\omega$, $K$, $\sigma$, $\pi_1$, $\pi_2$,
will be  essential for the main result of the paper,
but for the sake of perspective we will review a larger set.
Invariants of nets include, in particular,
classical invariants of curves, surfaces, and curves on surfaces,
which can be found in any textbook on classical differential geometry,
in particular~\cite{Spi-III}.
Relations among various invariants and the description how invariants change
under five discrete symmetries can be found in  Appendix~\ref{app:rel}.

Given a direction pair $([X_1],[X_2])$, the scalar differential invariants of
order $\le r$ can be constructed from
the Euclidean space metric and the derivatives $X_{i_s} \ldots X_{i_1} \ve r$, $1 \le s \le r$, as expressions that are invariant with respect to rigid
motions and multiplications~\eqref{eq:dp:transf}.
Higher-order scalar differential invariants can be obtained from
lower-order ones by applying the invariant differentiations $\widehat{X_i}$
(differentiations with respect to the arc length).

Following Sannia~\cite{GS},
expressions $E$ satisfying
$E' = f_1^{a_1} f_2^{a_2} E$ are called $(a_1,a_2)$-semi\-in\-vari\-ants.
Needless to say, invariants are synonymous to $(0,0)$-semiinvariants.

We start with invariants expressible in terms of the
fundamental coefficients~\eqref{eq:I-II}.
Under $X_i' = f_i X_i$, the latter transform as
\[
\I ij' = f_i f_j \I ij, \qquad
\II ij' = f_i f_j \II ij.
\]
Consequently, $\I ij$ and $\II ij$ are
$(\delta_{i1} + \delta_{j1}, \delta_{i2} + \delta_{j2})$-semiinvariants,
where $\delta_{ik}$ denotes the Kronecker symbol.

Observe that $\I ij$ are of order $1$, while $\II ij$ are of order $2$.
According to the appendix, Table~\ref{tab:inv}, there can be only one independent
first-order invariant,
for which we choose the non-oriented intersection angle $\omega$
determined by \looseness=-1
\[
\cos\omega = \frac{\I12}{\sqrt{\I11 \I22}},
\qquad
\sin\omega = \sqrt{\frac{\detI}{\I11 \I22}}
\]
between $0$ and $\pi$.
The oriented intersection angle between $0$ and $2 \pi$ can be defined analogously,
using $\sin\omega\,\ve n = \widehat{X_1} \ve r \times \widehat{X_2} \ve r$ to determine
$\sin\omega$.

Associated with the surface $S$ are two independent second-order invariants,
for which we choose the Gauss and the mean curvature
\[
K = \frac{\detII}{\detI},\qquad
H = \frac{\I 11\,\II 22 - 2 \I 12\,\II 12
 + \I 22\,\II 11}{\detI}.
\]
Associated with the curves of each family
are
the normal curvatures
\[
\nc{i} = \frac{\II ii}{\I ii},
\]
the geodesic curvatures
\[
\gc{i} = \frac{[X_i \ve r, X_i X_i \ve r, \ve n]}
  {\I ii^{3/2}}
\]
($[\ve u, \ve v, \ve w]$ denotes the triple product,
i.e., the oriented volume of the parallelepiped spanned by the vectors
$\ve u, \ve v, \ve w$),
the ordinary curvatures
\smash{$\curv{i} = \sqrt{\vphantom{|^2}\smash{\nc{i}^2 + \gc{i}^2}}$},
and the geodesic torsions~\cite{JKW} or~\cite[p.~165]{RvL}
\[
\gtor i = \frac{[X_i \ve r, \ve n, X_i \ve n]}{\I ii}
 = (-1)^i \frac{\I 12\,\II ii - \I ii\,\II 12}
     {\I ii \sqrt{\detI} }
\]
(ordinary torsions and normal torsions are of order 3).

Of utmost importance for us is the
{\it Schief curvature}
\begin{gather}\label{eq:sigma}
\sigma = \frac{X_2 X_1 \ve r \cdot \ve n}{\|X_1 \ve r \times X_2 \ve r\|}
 = \frac{[X_1 \ve r, X_2 \ve r, X_2 X_1 \ve r]}
  {[X_1 \ve r, X_2 \ve r, \ve n]^2} = \frac{\II 12}{\sqrt{\detI}},
\end{gather}
introduced by W.K. Schief~\cite[Section~3.1]{WKS-2007} as a continuous limit
of a curvature measure of discrete nets.
Considering an infinitesimal tetrahedron spanned by the net,
$\sigma$ turns out to be proportional
to the ratio of its volume
to the squared area
of its base, as well as to the ratio of its height to the area of
its base, see loc. cit. for the details.

\begin{Remark}\label{rem:sigma}
Obviously, conjugate nets are characterised by $\sigma = 0$, while
the wobbly (``wa\-ckelige'') nets of Sauer~\cite{RS-1933} are
characterised by admitting a $\sigma$-preserving isometry.
Schief \cite[Sec\-tion~2.2]{WKS-2007} related Chebyshev nets of constant $\sigma$
to the Pohlmeyer--Lund--Regge system.
\end{Remark}

Next we consider invariants expressible in terms of
$\I ij$ and $\widehat{X_i}$.
Firstly, for each $i = 1,2$, the derivative
\[
\omega_{,i} = \widehat{X_i} \omega
\]
of the
intersection angle with respect to the arc length is an invariant,
matching the description~of {\it courbure inclin\'ee}
by Aoust~\cite[I, Section~10]{AA}.

Secondly, the commutation relation
\begin{gather}\label{eq:iota}
[\widehat{X_i}, \widehat{X_j}] = \iota_j \widehat{X_i} + \iota_i \widehat{X_j}
\end{gather}
can be taken for the definition of second-order invariants $\iota_i$.
One easily checks that
\[
\iota_1 = \frac{\widehat{X_2} \I11}{2\,\I11}, \qquad
\iota_2 = -\frac{\widehat{X_1} \I22}{2\,\I22}.
\]

More classical are Bortolotti curvatures~\cite[equations~(1) and~(2)]{EB-1925},
which can be introduced in the following way.
Consider the covariant derivative $\nabla_{X_1} X_2$,
defined by the property that
$(\nabla_{X_1} X_2) \ve r$ is the projection
of the vector $X_1 X_2\,\ve r$ to the tangent space to $S$, at every
point.
Being tangent to the surface, $\nabla_{X_i} X_j$ can be expressed as a linear
combination \smash{$\Gamma^1_{ij} X_1 + \Gamma^2_{ij} X_2$} of~$X_1$,~$X_2$.%
\footnote{If $X_i = \partial/\partial x_i$, then $\Gamma^k_{ij}$
become the usual Christoffel symbols.}\,\footnote{Contrary to Christoffel symbols, $\Gamma^k_{ij} \ne \Gamma^k_{ji}$ in general.}\,\footnote{By the way, $\Gamma^1_{21}$, $\Gamma^2_{12}$ are not semiinvariants,
while
$\Gamma^{2}_{11}$, $\Gamma^{1}_{22}$ are related to the geodesic curvatures,
see~\cite{EB-1925}.}
By Cramer's rule, explicit expressions for \smash{$\Gamma^1_{12}$}, \smash{$\Gamma^2_{21}$}
are
\begin{gather}
\Gamma^1_{12} = \frac1{\detI}
 \left|\begin{matrix}
  X_1 \ve r \cdot X_1 X_2 \ve r & X_1 \ve r \cdot X_2 \ve r \\
  X_2 \ve r \cdot X_1 X_2 \ve r & X_2 \ve r \cdot X_2 \ve r
 \end{matrix}\right|,
\qquad
\Gamma^2_{21} = \frac1{\detI}
 \left|\begin{matrix}
  X_1 \ve r \cdot X_1 \ve r & X_1 \ve r \cdot X_2 X_1 \ve r \\
  X_2 \ve r \cdot X_1 \ve r & X_2 \ve r \cdot X_2 X_1 \ve r
 \end{matrix}\right|.\label{eq:Gamma}
\end{gather}
It is easily seen that $\Gamma^1_{12}$ is a $(0,1)$-semiinvariant, while
$\Gamma^2_{21}$ is a $(1,0)$-semiinvariant.
Hence,
\begin{gather}\label{eq:pi}
\pi_1 = \frac{\Gamma^1_{12}}{\sqrt{\I 22}}, \qquad
\pi_2 = \frac{\Gamma^2_{21}}{\sqrt{\I 11}}
\end{gather}
are invariants.
Up to signs,
$
\pi_1 \sin\omega$,
$\pi_2 \sin\omega$
coincide with the aforementioned
Bortolotti curvatures~\cite[equations~(1) and~(2)]{EB-1925}.
Related to them are also the Chebyshev curvature and the Chebyshev vector,
see~\cite{KHW-1940} and~\cite[Section~23]{VIS}, which we omit.

\section{Chebyshev nets}\label{sect:Chebyshev}

Originally introduced to model woven fabrics conforming to a
body~\cite{PLC,Ghy},
Chebyshev nets have important applications and are subject to
active research till today~\cite{AMD,H-O-O,M-M-2017,SF-C-BC-V}.
As the most exciting architectural application,
Chebyshev nets model elastic timber structures (gridshells) obtained by
buckling a flat straight rectangular grid connected by
joints~\cite{IL-2015}.
A manifestly invariant characterisation of Chebyshev nets is the
{\it curvilinear parallelogram condition} (opposite
sides of curvilinear quadrilaterals formed by pairs of curves of each family
have the same length),
see Bianchi~\cite[Section~379]{LB-II} or Darboux~\cite[Section~642]{GD-III}.

\begin{Proposition}\label{prop:Chebyshev}
The following statements about a net and the corresponding direction pair
$([X_1], [X_2])$ are equivalent\/{\rm:}
\begin{enumerate}\itemsep=0pt
\item[\rm(i)]
the family parameters $x$, $y$ can be chosen in such a way that the first
fundamental form is
\begin{gather}\label{eq:Cheb:Iff}
{\rm d} x^2 + 2 \cos\f\,{\rm d} x\,{\rm d} y + {\rm d} y^2
\end{gather}
{\rm(}the Chebyshev parameterisation\/{\rm;}
$\omega$ coincides with the intersection angle invariant\/{\rm);}
\item [\rm(ii)]
all unit vectors in one direction of the net are parallel along all curves
in the other direction, i.e.,
\[
\nabla_{X_1} \widehat{X_2} = 0, \qquad \nabla_{X_2} \widehat{X_1} = 0
\]
{\rm (}see Bianchi~{\rm\cite{LB-1922});}
\item[\rm(iii)]
the unit representatives commute, i.e.,
\[
\bigl[\widehat{X_1}, \widehat{X_2}\bigr] = 0;
\]
\item[\rm(iv)] the invariants $\iota_i$ vanish, that is,
\[
\iota_1 = 0 = \iota_2, \qquad \text{i.e.,} \qquad X_2 \I11 = 0 = X_1 \I22;
\]
\item[\rm(v)] the Bortolotti curvatures~\eqref{eq:pi} vanish, that is,
\[
\pi_1 = 0 = \pi_2;
\]
\item[\rm(vi)] the geodesic curvatures satisfy
\[
\gc{1} = -\widehat{X_1} \omega, \qquad \gc{2} = \widehat{X_2} \omega
\]
{\rm(}see {\rm\cite[equation~(4.7)]{YM-2017})}.
\end{enumerate}
\end{Proposition}

\begin{proof}
(i) $\Rightarrow$ (ii).
If (i) holds, then both $\partial/\partial x$ and $\partial/\partial y$ are unit
vectors.
The Bianchi condition~(ii) can be verified by the straightforward calculation
of the covariant derivatives.

(ii) $\Rightarrow$ (iii).
If (ii) holds, then $[\widehat{X_1}, \widehat{X_2}]
 = \nabla_{\widehat{X_1}} \widehat{X_2} - \nabla_{\widehat{X_2}} \widehat{X_1} = 0$.

(iii) $\Rightarrow$ (i).
If (iii) holds, then one can choose coordinates $x$, $y$ in such a way that
$\partial/\partial x = \widehat{X_1}$ and $\partial/\partial y = \widehat{X_2}$.
Being equal to the squared lengths of the vectors $\widehat{X_i} \ve r$,
the metric coefficients at ${\rm d} x^2$ and ${\rm d} y^2$ are equal to $1$.

(iii) $\Leftrightarrow$ (iv) is obvious by formula~\eqref{eq:iota},
which defines $\iota_i$.

(iv) $\Leftrightarrow$ (v) is obvious from identities~\eqref{eq:iotapi} in
Appendix~\ref{app:rel}.

(v) $\Leftrightarrow$ (vi) is obvious from identities~\eqref{eq:pigc} in
Appendix~\ref{app:rel}.
\end{proof}

\begin{Remark}Another criterion is the vanishing of the Chebyshev vector~\cite[Section~67]{APN}
or~\cite[Section~55]{VIS}.
Yet different criteria can be found in~\cite{WCG-1932,SF-C-BC-V,WKS-2007,CEW}.
\end{Remark}

\begin{Remark}\label{rem:bisect}
Associated with every Chebyshev parameterisation~\eqref{eq:Cheb:Iff} is the
{\it isodiagonal} parameterisation
(\cite[Section~1]{AV-1869} or \hbox{\cite[Section~678]{GD-III}}) by
$u = x + y$, $v = x - y$.
In terms of $u$, $v$, the metric is
\[
{\rm I} = \cos^2 \tfrac12 \omega \dif u^2 + \sin^2 \tfrac12 \omega \dif v^2.
\]
\end{Remark}

\section{Integrable nets}\label{sect:integrability}

The literature on soliton geometries is very extensive,
but authors (except~\cite{Sym}) seem reluctant to define them in any other way
than by means of examples.
In this section we attempt to give a definition, which covers both surfaces and nets
(Definition~\ref{def:integ}).

Integrability is understood in the conventional sense of
soliton theory~\cite{Bob,G-T,R-S,Sym}.
The integrability criterion is the existence
of a zero-curvature representation (ZCR)~\cite{Z-S-1974}
\[
D_y A - D_x B + [A, B] = 0,
\]
where, firstly, $A$, $B$ are elements of a finite-dimensional and
non-solvable matrix Lie algebra that cannot be reduced to a solvable
one by gauge transformations and,
secondly, $A$, $B$ depend on a (spectral) parameter that is not removable
by gauge transformation.
A gauge transformation by means of a gauge matrix $H$ is the
correspondence
\begin{gather*}
A' = D_x H \cdot H^{-1} + H \cdot A \cdot H^{-1}, \qquad
B' = D_y H \cdot H^{-1} + H \cdot B \cdot H^{-1}.
\end{gather*}
For simple criteria of reducibility and removability, see~\cite{MM-rzcr,MM-spp}.

For both surfaces and nets, we require integrability of the
Gauss--Mainardi--Codazzi system.
The system is, in compact form~\cite{Spi-III,Sym},
\begin{gather}\label{eq:GMC}
R_{ijkl} = \II jk \II il - \II ik \II jl, \qquad
\mathrm{II}_{ij;k} = \mathrm{II}_{ik;j}
\end{gather}
($R_{ijkl}$ is the Riemann tensor and the semicolon denotes the covariant derivatives).
We also recall that the Gauss--Mainardi--Codazzi equations are the
compatibility conditions of the Gauss--Weingarten system
\begin{gather}\label{eq:GW}
\ve r_{,ij} = \Gamma^k_{ij} \ve r_{,k} + \II ij \ve n,
\qquad
\ve n_{,i} = {\rm II}^k_i \ve r_{,k},
\end{gather}
which describes the immersed surfaces and their normals
($\Gamma^k_{ij}$ are the Christoffel symbols and~the index
$k$ in ${\rm II}^k_i$ is raised by the metric $\I ij$).
In expanded form, the Gauss--Mainardi--Codazzi system consists
of three partial differential equations on six unknowns
$\I ij$, $\II ij$, and can be found in all standard textbooks on surface
geometry.

Besides integrability, another key point is the geometric characterisability
of the class.
The three partial differential equations on six unknowns can be supplemented
with as much as three other conditions
(or more if auxiliary functions are introduced).
Normally, two conditions (usually algebraic) are spent on specifying a particular
parameterisation,
leaving room for one condition to characterise the class.

To characterise a geometric class of surfaces (nets) in Euclidean space,
the condition must be invariant with respect to Euclidean motions and arbitrary
reparameterisations of surfaces (nets).
In other words, there must exist a formulation of the condition in terms
of differential invariants of surfaces (nets), at least in principle.
Therefore, it seems natural to define integrable classes in the following way,
suitable for specifying classification problems.

\begin{Definition}
\label{def:integ}
A class of surfaces (nets) is called {\it integrable} if it can be determined by a~condition written in terms of differential invariants of surfaces (nets) and the
Gauss--Mainardi--Codazzi system augmented with this condition
is integrable in an appropriate parameterisation.
\end{Definition}

\begin{Proposition}
If a class of nets is integrable,
then so is the class of supported surfaces.
\end{Proposition}

\begin{proof}Obvious from the definition.
\end{proof}

The appropriate parameterisation the definition refers to
should exist for every member of the class.
Its purpose is to make the whole system determined.
For instance, the parameterisation may be principal for generic surfaces,
asymptotic for hyperbolic surfaces, Chebyshev for Chebyshev nets, etc.
However, experience shows that if a system is integrable in one parameterisation,
then it is integrable in any other, even in a general one
(in which case the whole system is underdetermined).
This may be related to the fact that the zero curvature representation is also a geometric
notion, which can be understood as a matrix-Lie-algebra-valued 1-form
$\alpha = A \dif x + B \dif y$
satisfying $\dif \alpha = \frac12 [\alpha,\alpha]$,
and the gauge transformation as
$\alpha' = \dif H \cdot H^{-1} + H \cdot \alpha \cdot H^{-1}$.

Integrable classes of nets have been with us since the dawn of differential geometry
of surfaces.
For principal conformal nets, see Remark~\ref{rem:isotherm} below.
To name others, conjugate nets are connected with the
Laplace--Darboux integrability \cite{GD-I,BGK-1993}.
Moreover, classical integrable geometries include
integrable curve evolutions \cite{HH-1972,GLL-1977,M-B-2008,R-S,S-R-1999},
which form integrable nets if completed with the evolution trajectories.
Furthermore, integrable foliations of surfaces by curves
\cite{C-T-1981,WKS-2003,CT-1988}
can be completed to integrable nets by the orthogonal curves.
Apparently, already a review of the known cases would be a formidable task, not
speaking about their invariant characterisations.

A systematic search for integrable classes of nets can be performed in the same
manner as the search for integrable classes of surfaces.
A natural way is to incorporate a non-removable spectral parameter into
the $\mathfrak{so}(3)$-valued zero-curvature representation induced by the
Gauss--Weingarten system~\cite{Sym},
either by the symmetry method \cite{JC-1993,L-S-T,DCF-LAOS}
or by the more powerful cohomological method~\cite{B-M-2009}.
\looseness-1

It is worth mentioning that classification results for integrable nets may also include
integrable surfaces equipped with the nets in question.
For example, linear Weingarten surfaces appeared in the classification
of integrable classes of Chebyshev parameterisations\footnote{Integrable classes
of parameterisations can be introduced by Definition~\ref{def:integ} stripped of the
invariance requirement.} in~\cite[Section~2]{K-M}.

\begin{Remark}\label{rem:isotherm}
Integrable classes of nets and integrable classes of surfaces mutually correspond
(think of the class of all surfaces capable of carrying the nets).
Therefore, classification of integrable surfaces and
classification of integrable nets are interrelated, but in a complicated way.\looseness=-1

For instance, isothermic surfaces and principal conformal nets
(meaning nets generated by principal conformal parameterisations)
\cite{C-G-S,MT-2017}
determine each other uniquely and
the study of isothermic surfaces is the same thing as the
study of principal conformal nets.
It can be easily seen that principal nets are characterised by the vanishing of
$\cos\omega$ and either of $\sigma$, $\gtor1$,  $\gtor2$, which are of order 1 and 2,
respectively,
whereas conformal nets are characterised by the vanishing of $\cos\omega$
and \smash{$\widehat{X_1}\,\gc1 + \widehat{X_2}\,\gc2$}, which are of order 1 and 3, respectively.
On the other hand, the lowest-order nontrivial surface invariant
vanishing for all isothermic surfaces is
$(k_1  - 2 k_2)_{,12} + (k_2 - 2 k_1)_{,21}$,
which is of order 4
($k_i$~are the principal curvatures and comma denotes
differentiation with respect to the arc length in principal directions).
Therefore, principal conformal nets appear earlier (at lower order)
in the classification of nets than isothermic surfaces in the classification of surfaces.
\end{Remark}

As a rule, if a net is integrable, then so are the various {\it derived} nets
(on the same or another surface) obtained by geometric constructions.
Thus, a complete classification of integrable classes (to a certain order of
invariants), if such a goal were achievable, would consist of a rather complex
interconnected (and infinite) network.
However, invariant description of many derived nets will be of higher order
than that of the net they are derived from,
often far out of reach of presently available classification methods.
Classification efforts will most likely spot only the
integrable classes on the ``border'', while the derived nets will allow to penetrate
deeper into the ``integrable~region''.

Let us, finally, remark that one may also look for integrable parameterisations
of a given surface, requiring the integrability of the system
to obtain such a parametrisation
(for instance, the Servant equations,
see in the beginning of the next section).
This is, however, a different problem.

\section{Integrable Chebyshev nets}\label{sect:int:Cheb}

Voss~\cite{AV-1869} obtained large classes of explicit Chebyshev nets,
among others on surfaces of revolution;
he also proved that Chebyshev nets on the sphere correspond to solutions
of the sine-Gordon equation~\cite[Section~3]{AV-1869}.
For pictures, see~\cite{H-O-O,M-T-1956};
the work~\cite{H-O-O} also addresses Chebyshev nets of class~$C^1$.
Given a surface metric, obtaining general Chebyshev nets is possible by
solving the Servant equations \cite[equation~(3)]{MS-1904}, which are, however,
not always integrable.
Integrable are also special Chebyshev nets that can be found
according to \cite[Section~2.2]{WKS-2007},
cf.\ Remark~\ref{rem:sigma}.

In the earlier paper~\cite[Section~2]{K-M}, we looked for integrable
Gauss--Mainardi--Codazzi systems in Chebyshev parameterisation.
Our result consisted of five classes,\footnote{Chebyshev nets on linear Weingarten surfaces have been studied in~\cite{DL-2023}.}
including Case~2 specified by the linear relation
\begin{gather}\label{eq:CondII}
\mu K + \kappa\frac{\II12}{\sin\omega} + \lambda = 0,
\end{gather}
where $\mu$, $\kappa$, $\lambda$ are real constants, $K$ is the Gauss curvature,
$\II12$ is the coefficient of the second fundamental form with respect to the
Chebyshev parameterisation, and $\omega$ is the intersection angle.
As the parameterisation-dependent term $\II12/{\sin\omega}$ in
formula~\eqref{eq:CondII} coincides with the
Schief curvature~\eqref{eq:sigma}
(since Chebyshev parameterisations satisfy $\I11 = 1 = \I22$),
we see that condition~\eqref{eq:CondII} can be rewritten as
\begin{gather}\label{CondII}
\m K + \k \sigma + \l = 0,
\end{gather}
where $\m$, $\k$, $\l$ are arbitrary constants.
Manifestly, condition~\eqref{CondII} specifies a geometric
class of nets.
We already know from~\cite{K-M} that
the corresponding Gauss--Mainardi--Codazzi system is integrable
(has a ZCR).
Hence, condition~\eqref{CondII}
determines an integrable class of nets according to Definition~\ref{def:integ}.

Topologically, the ``space'' of conditions $\m K + \k \sigma + \l = 0$
is the projective space $\mathbb RP^2$ (a~sphere with identified antipodal points),
see Figure~\ref{fig:world}.
The discrete symmetries $T_{-1}, \dots, T_2$ (see Table~\ref{tab:Ts:inv} in the appendix)
change the sign of $\sigma$, that is, the sign of $\k$,
identifying $\m K + \k \sigma + \l = 0$ with $\m K - \k \sigma + \l = 0$.
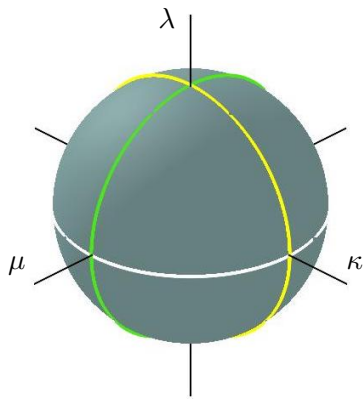
\begin{figure}[th]\centering
\unitlength = 1cm
\begin{picture}(5,5.5)
\put(0.5,0){\includegraphics[scale = 0.32]{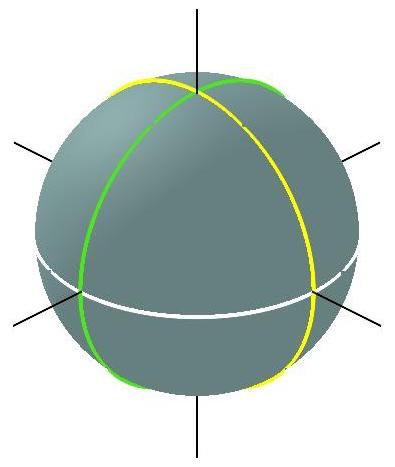}}
\put(2.3,5){$\l$}
\put(0.3,1.8){$\m$}
\put(4.8,1.8){$\k$}
\end{picture}
\caption{The space of conditions \protect$\m K + \k \sigma + \l = 0\protect$. Antipodal points coincide.}
\label{fig:world}
\end{figure}

\begin{Remark}
When at least one of $\m$, $\k$ is zero,
the Chebyshev nets satisfying condition~\eqref{CondII}
fall into one of the following classes:
\begin{enumerate}\itemsep=0pt
\item
If $\k = 0$, $\m \ne 0$ (the green circle in Figure~\ref{fig:world}),
then condition~\eqref{CondII} implies the constancy of~$K$.
Thus, we arrive at surfaces of constant Gaussian curvature equipped
with an arbitrary Chebyshev net, including developable surfaces
(the intersection of green and white circle).

\item  If $\mu = 0$ and $\k \ne 0$ (the yellow circle in Figure~\ref{fig:world}),
then condition~\eqref{CondII} implies
the constancy of the Schief curvature $\sigma$, which is the situation
explored in Schief~\cite[Section~2.2]{WKS-2007}.
One obtains the equation $\ve r_{xy} = \sigma \ve r_x \cdot \ve r_y$,
identifiable with the integrable Lund--Regge system.
For finite-gap solutions, see Shin~\cite{HJS}.
If, moreover, $\l = 0$ (the intersection of yellow and white circle), then $\sigma = 0$.
This yields the well-understood class of translation surfaces~\mbox{\cite[Sections~81 and~82]{GD-I}},
i.e., solutions of the equation $\ve r_{xy} = 0$.
\end{enumerate}
\end{Remark}

We see that the cases of $\m = 0$ or $\k = 0$ (the green and the yellow circle)
have already been sufficiently understood.
Therefore, we may assume that $\m \ne 0 \ne \k$ in what follows.
Dividing condition~\eqref{CondII} by $\mu \ne 0$ is equivalent to setting
$\m = 1$,
which we assume henceforth.

\begin{Remark}\label{rem:CondII:ways}
Using identities listed in Appendix~\ref{app:rel},
condition~\eqref{CondII} can be rewritten in different ways, for instance
\[
\frac{\nc1^2 + \gtor1^2 - \kappa\,\gtor1 - \lambda}{\nc1}
 = \frac{\nc2^2 + \gtor2^2 + \kappa\,\gtor2 - \lambda}{\nc2}
\]
(a relation between two curve invariants on the surface) or
\[
\cot\omega + \cot\omegaIII
 = -2 H \frac{ K + \lambda}{\kappa K}
\]
(a relation among two angle invariants and a surface invariant).
\end{Remark}

From now on, until otherwise stated, we use the Chebyshev
parameterisation, i.e., we consider the first fundamental
form~\eqref{eq:Cheb:Iff}, leaving the second fundamental form arbitrary.
We assume that $\sin\f \ne 0$ henceforth, i.e., we assume that all points
are nonsingular in the sense of Definition~\ref{def:net}.

Let us introduce variables $\h ij$ by
\[
\II ij = \h ij \sin\f.
\]
In terms of $\h ij$, the Gauss\ and the Schief curvatures are simply
\[
K = \h11 \h22 - \h12^2 = \det h,
\qquad
\sigma = \h12,
\]
while condition~\eqref{CondII}  becomes
\begin{gather}\label{CondII:h}
\m \bigl(\h11 \h22 - \h12^2\bigr) + \k \h12 + \l = 0.
\end{gather}
The Gauss--Weingarten system is
\begin{gather}\label{GW}
\begin{split}
&\ve r_{xx} = \h11\sin\f\,\ve n + \f_x\cot\f\,\ve r_x - (\f_x/{\sin\f})\,\ve r_y, \\
&\ve r_{xy} = \h12 \sin\f\,\ve n, \\
&\ve r_{yy} = \h22 \sin\f\,\ve n + \f_y \cot\f\,\ve r_y - (\f_y/{\sin\f})\,\ve r_x, \\
&\ve n_x = \frac{\h12 \cos\f - \h11}{\sin\f}\,\ve r_x
 + \frac{\h11 \cos\f - \h12}{\sin\f}\,\ve r_y , \\
&\ve n_y = \frac{\h22 \cos\f - \h12}{\sin\f}\,\ve r_x
 + \frac{\h12 \cos\f - \h22}{\sin\f}\,\ve r_y,
\end{split}
\end{gather}
the Gauss--Mainardi--Codazzi equations
(the compatibility conditions of the Gauss--Weingarten system) being
\begin{gather}
\begin{split}
& \f_{xy} + K \sin\f = 0,  \\
&\h{11}{,y} = \h{12}{,x} - \h11 \f_y \cot\f + \h22 \f_x/{\sin\f},  \\
&\h{12}{,y} = \h{22}{,x} - \h11 \f_y/{\sin\f} + \h22 \f_x \cot\f.
\end{split}\label{GMC}
\end{gather}
These systems should be completed with condition~\eqref{CondII:h}.
We do this by solving~\eqref{CondII:h} for $\h22$ and inserting
\begin{gather}\label{eq:h22}
\h22 = \frac{\m \h12^2- \k\h12 -\l}{\m \h11}
\end{gather}
into~\eqref{GW} and~\eqref{GMC}.

\section{Vector conservation laws}\label{sect:vcl}

In this section, we look for vector conservation laws of the form
$\ve P \dif x + \ve Q \dif y$, where
$\ve P$, $\ve Q$ are linear combinations of $\ve r_x$, $\ve r_y$, $\ve n$ such that
\[
D_y \ve P - D_x \ve Q = 0
\]
holds as a consequence of the Gauss--Mainardi--Codazzi
equations~\eqref{GMC} and the Gauss--Wein\-gar\-ten equations~\eqref{GW}
under condition~\eqref{CondII}.
For every vector conservation law, we define the associated vector potential
$\ve w$ to be a vector satisfying
$\dif \ve w = \ve P \dif x + \ve Q \dif y$, that is,
$\ve w_x = \ve P$, $\ve w_y = \ve Q$.
The vector conservation law is said to be trivial if the corresponding
potential $\ve w$ can be found among the local functions as a
linear combination of $\ve r_x$, $\ve r_y$, $\ve n$, the coefficients being functions
of $x$, $y$, $\f$, $\h11$, $\h12$ and their derivatives.

Finding vector conservation laws is no harder than finding scalar ones.
In our case, the main obstacle is that the Gauss--Weingarten system is overdetermined
and, therefore,
we cannot use the correspondence between conservation laws and cosymmetries.
Wolf's~\cite{Wolf} comparison of four approaches to computation of
conservation laws indicates that the method that is most likely to lead to
an answer, is the following (the third) one.

Let $W_i = 0$, $i = 1, \dots, 3$, be individual equations of the
Gauss--Mainardi--Codazzi system~\eqref{GMC} and
$\ve W_i = 0$, $i = 1, \dots, 5$,
individual equations of the Gauss--Weingarten system~\eqref{GW}.
For further reference,
\begin{gather*}
W_2 = -\h{11}{,y} + \h{12}{,x} - \h11 \f_y \cot\f + \h22 \f_x/{\sin\f}, \\
W_3 = -\h{12}{,y} + \h{22}{,x} - \h11 \f_y/{\sin\f} + \h22 \f_x \cot\f, \\
\ve W_1 = -\ve r_{xx} + \h11 \sin\f\,\ve n +  \f_x\cot\f\,\ve r_x - (\f_x/\sin\f)\,\ve r_y, \\
\ve W_2 = -\ve r_{xy} + \h12\sin\f\,\ve n, \\
\ve W_3 = -\ve r_{yy} + \h22\sin\f\,\ve n + \f_y \cot\f\,\ve r_y
 - (\f_y/\sin\f)\,\ve r_x,
\end{gather*}
(we omit $W_1$, $\ve W_4$ and $\ve W_5$, which we shall not need explicitly).
Then we can write
\[
D_y \ve P - D_x \ve Q = \sum \ve C_i W_i + \sum C_i \ve W_i,
\]
for suitable characteristics $\ve C_1$, $\ve C_2$, $\ve C_3$ and $C_1, \dots, C_5$.
Applying the Euler--Lagrange operator
\[
\frac{\delta}{\delta z} = \sum_J (-D)_J \frac{\partial}{\partial z}
\]
with $z$ running through all
dependent variables $z = \ve r, \ve n, \f, \h11, \h12$, we get
\begin{gather}\label{EL}
0 = \sum_J (-D)_J \frac{\partial}{\partial z}
   \Bigl(\sum \ve C_i W_i + \sum C_i \ve W_i\Bigr),
 \qquad z = \ve r, \ve n, \f, \h11, \h12.
\end{gather}
These are five equations on the eight unknowns
$\ve C_1$, $\ve C_2$, $\ve C_3$, $C_1, \dots, C_5$.
Three ignorable solutions correspond to the trivial conservation laws
$\dif\ve n$,
$\dif\ve r_x$,
$\dif\ve r_y$.
A non-ignorable solution to~\eqref{EL}~is\looseness=-1
\begin{alignat*}{3}
&C_1 = \h22, \qquad &&       \ve C_1 = 0,& \\
&C_2 = 2\k - 2\h12, \qquad && \ve C_2 = \ve r_y, &\\
&C_3 = \h11, \qquad && \ve C_3 = -\ve r_x,&
\end{alignat*}
valid if and only if $\l = 0$.
This leads us to the following proposition.

\begin{Proposition}\label{prop:vcl}
Assuming $\sin\f \ne 0$, expressions
\[
\ve P = (\h12 - \k)\,\ve r_x - \h11\,\ve r_y, \qquad
\ve Q = \h22\,\ve r_x + (\k - \h12)\,\ve r_y
\]
 are components of a vector conservation law if and only if $\l = 0$.
\end{Proposition}

\begin{proof}
It is straightforward to check that
$D_y \ve P - D_x \ve Q = 2\l \sin\f\, \ve n$,
which is zero if and only if $\lambda = 0$.
\end{proof}

The vanishing of $\l$ (the white circle in Figure~\ref{fig:world})
means that the Schief curvature $\sigma$ is proportional to the Gauss\
curvature $K$.
After the concordance of the two measures, we introduce the following terminology
(applicable to arbitrary nets, non necessarily Chebyshev ones).

\begin{Definition}
Nets satisfying $K = \k \sigma$, $\k \in \mathbb R$,
will be called {\it concordant nets}.
\end{Definition}

By Remark~\ref{rem:CondII:ways} and  formula~\eqref{eq:cotphiIII},
an equivalent formulation of concordance is
\[
{\cot\omegaIII} + {\cot\omega} = -2 H/\k.
\]

\section{From concordant nets to pairs of pseudospherical surfaces}
\label{sect:ccnets->PS}

In this section, $x$, $y$ continue to denote the Chebyshev parameters.

Following Proposition~\ref{prop:vcl},
let $\ve m$ denote the vector potential satisfying
\begin{gather}\label{m}
\begin{split}
&\ve m_x = (\h12 - \k)\,\ve r_x - \h11\,\ve r_y, \\
&\ve m_y = \h22\,\ve r_x + (\k - \h12)\,\ve r_y.
\end{split}
\end{gather}
The vector $\ve m$ is crucial in what follows.

\begin{Definition}
We define the {\it associated surfaces\/} $S^+$, $S^-$ of a concordant
net by the parameterisations
\begin{gather}\label{eq:assoc}
\ve r^+ = \ve r + \ve m/\k, \qquad
\ve r^- = \ve r - \ve m/\k.
\end{gather}
\end{Definition}

\begin{Theorem}\label{thm:ccn->2ps}
Consider  a concordant Chebyshev net satisfying  $K = \k \sigma$.
Then
\begin{enumerate}\itemsep=0pt
\item[{\rm (i)}]  the associated surfaces  $\ve r^+$, $\ve r^-$ are regular
wherever $\sigma \ne 0$ and $\sin\omega \ne 0${\rm;}

\item[{\rm (ii)}]  $\ve r^+$, $\ve r^-$ are pseudospherical of the Gauss\ curvature $-\kappa^2${\rm;}

\item[{\rm (iii)}]  all three surfaces $\ve r^+$, $\ve r^-$, $\ve r$ have one and the same normal
vector $\ve n$ at the corresponding points\/{\rm;}

\item[{\rm (iv)}]  assuming that $x$, $y$ are Chebyshev parameters,
$[\partial/\partial x]$ and $[\partial/\partial y]$ are
asymptotic directions for $\ve r^+$ and $\ve r^-$, respectively.
\end{enumerate}
\end{Theorem}

\begin{proof}
Obviously from formulas~\eqref{m} and~\eqref{eq:assoc},
$\ve n$ is orthogonal to both $\ve r^\pm_x$ and $\ve r^\pm_y$,
and the third statement follows.

Computing the components of the corresponding fundamental forms ${\rm I}^\pm$ and ${\rm II}^\pm$,
we get
\begin{gather*}
\k^2 {\rm I}^+_{11} = \h11^2 - 2 \cos\f\,\h11\h12 + \h12^2, \\
\k^2 {\rm I}^+_{12} = \h11 (\h12 - 2 \k)
 - 2 \cos\f\,\h12 \bigl(\h12 - \tfrac32 \k\bigr) + \h12 \h22, \\
\k^2 {\rm I}^+_{22} = (\h12 - 2 \k)^2 - 2 \cos\f\,(\h12 - 2 \k)\h22 + \h22^2
\end{gather*}
and, symmetrically,
\begin{gather*}
\k^2 {\rm I}^-_{11} = \h11^2 - 2 \cos\f\,(\h12 - 2 \k)\h11
 + (\h12 - 2 \k)^2, \\
\k^2 {\rm I}^-_{12} = \h11 \h12 - 2 \cos\f\,\h12 \bigl(\h12 - \tfrac32 \k\bigr)
 + (\h12 - 2 \k)\h22, \\
\k^2 {\rm I}^-_{22} = \h12^2 - 2 \cos\f\,\h12\h22 + \h22^2.
\end{gather*}
Then $\det {\rm I}^\pm = (\sigma/\k)^2 \sin^2\omega$ is nonzero
wherever $\sigma \ne 0$ and $\sin\omega \ne 0$, which proves the first statement.

Concerning ${\rm II}^\pm$ we have
\begin{alignat*}{4}
&{\rm II}^+_{11} = \hm 0,\qquad && {\rm II}^+_{12} = \hm \sin\f\,\h12,\qquad &&
{\rm II}^+_{22} = 2 \sin\f\,\h22, &\\
&{\rm II}^-_{11} = -2 \sin\f\,\h11,\qquad && {\rm II}^-_{12} = -\sin\f\,\h12,\qquad &&
{\rm II}^-_{22} = 0.&
\end{alignat*}
The vanishing of ${\rm II}^+_{11}$ and ${\rm II}^-_{22}$ reveals the asymptotic directions
$\partial/\partial x$ and $\partial/\partial y$, which proves the fourth statement.

Using equation~\eqref{CondII}$|_{\l = 0}$, we get
\[
K^+ = \frac{\det {\rm II}^+}{\det {\rm I}^+} = -\k^2,
\qquad
K^- = \frac{\det {\rm II}^-}{\det {\rm I}^-} = -\k^2,
\]
which proves the second statement.
\end{proof}

To equip the surfaces $S^+$, $S^-$ with the
asymptotic Chebyshev parameterisations,
we employ the mean curvatures, which are easily seen to be
\begin{gather}\label{Hab}
H^+ = \frac{\h12 \sin\f}{\h12 \cos\f - \h11},
\qquad H^- = \frac{\h12 \sin\f}{\h22 - \h12 \cos\f}.
\end{gather}
Here and in what follows,
$\h22 = \bigl(\h12^2- \k\h12\bigr)/\h11$
by formula~\eqref{eq:h22}.

\begin{Proposition} \label{prop:xieta}
Denote
\[
\omegaa = -{\arctan} \frac{\k}{H^+},
\qquad
\omegab = {\arctan} \frac{\k}{H^-},
\]
where $H^+$ and $H^-$ are given by formulas~\eqref{Hab}.
Let  $\xi^- = x$, $\eta^+ = y$.
In the notation from the proof of Theorem~{\rm \ref{thm:ccn->2ps}},
define $\xi^+$ and $\eta^-$ by compatible equations
\begin{gather}\label{eq:x'}
\xi^+_x = \sqrt{{\rm I}^+_{11}}
 = \frac{1}{\k} \sqrt{\h11^2 - 2\h11\h12 \cos\f + \h12^2},
\qquad
\xi^+_y = \frac{\h22}{\h12} \xi^+_x
\end{gather}
and
\begin{gather}\label{eq:y'}
\eta^-_x = \frac{\h11}{\h12} \eta^-_y,
\qquad
\eta^-_y = \sqrt{{\rm II}^-_{22}}
 = \frac{1}{\k} \sqrt{\h12^2 - 2\h12\h22 \cos\f + \h22^2},
\end{gather}
respectively.
Then $\xi^+$, $\eta^+$ and $\xi^-$, $\eta^-$ are the corresponding asymptotic
Chebyshev parameters on $\ve r^+$ and~$\ve r^-$, while
$\phi^+$ and $\phi^-$ are the corresponding Chebyshev angles.
\end{Proposition}

\begin{proof}
One can check that systems~\eqref{eq:x'} and~\eqref{eq:y'} are indeed compatible and
\begin{gather*}
{\rm I}^\pm = (\dif\xi^\pm)^2 + 2 \cos\phi^\pm \dif\xi^\pm \dif\eta^\pm + (\dif\eta^\pm)^2, \\
{\rm II}^\pm = \pm 2 \k \sin\phi^\pm \dif\xi^\pm \dif\eta^\pm
\end{gather*}
by straightforward computation.
This implies both statements.
\end{proof}

\begin{Corollary} 
In the notation from Proposition~\rmref{prop:xieta},
\[ 
\phi^\pm_{\xi^\pm \eta^\pm} = \k^2 \sin\phi^\pm,
\]
meaning that $\phi^\pm(\xi^\pm,\eta^\pm)$ are solutions of the
sine-Gordon equation.
\end{Corollary}

\begin{Proposition}\label{prop:D:xieta}
In the notation from Proposition~\rmref{prop:xieta},
the coordinate vector fields corresponding to $\xi^\pm$, $\eta^\pm$ are
\begin{gather*}
D_{\xi^+} = \frac{\k}
  {\sqrt{h_{11}^2 - 2 h_{11} h_{12} \cos\omega + h_{12}^2}} D_x,
 \qquad D_{\eta^+} = -\frac{h_{22}}{h_{12}} D_x + D_y,
\\
D_{\xi^-} = D_x - \frac{h_{11}}{h_{12}} D_y,
 \qquad D_{\eta^-} = \frac{\k}
  {\sqrt{\h12^2 - 2\h12\h22 \cos\f + \h22^2}} D_y.
\end{gather*}
\end{Proposition}

\begin{proof}
By straightforward verification of
$D_{\xi^\pm} \xi^\pm = 1$,  $D_{\xi^\pm} \eta^\pm = 0$,
$D_{\eta^\pm} \xi^\pm = 0$, $D_{\eta^\pm} \eta^\pm = 1$,
and $[D_{\xi^\pm}, D_{\eta^\pm}] = 0$.
\end{proof}

It is well known that the asymptotic Chebyshev net on
a~pseudospherical surface induces a~Chebyshev net on the Gauss
sphere (and vice versa).
Consequently, the pair $\ve r^\pm$ induces a pair of such nets.
Their relative position depends on the angle $\omega$ in a very simple way.

\begin{Proposition}
\label{prop:Chpair}
In the notation from Proposition~\rmref{prop:D:xieta},
\begin{enumerate}\itemsep=0pt
\item[{\rm (i)}] the fields $D_{\xi^\pm}$, $D_{\eta^\pm}$ induce a pair of Chebyshev
nets on the unit sphere{\/\rm;}

\item[{\rm (ii)}]
the oriented angle $\angle(D_{\xi^-} \ve n, D_{\eta^+} \ve n)$
equals $\pi + \omega$.
\end{enumerate}
\end{Proposition}

\begin{proof}
The tangent vectors to the Gauss sphere are
\begin{gather*}
D_{\xi^+} \ve n
 = \frac{\k}{\sin\omega}\,
   \frac{(\h11 - \h12 \cos\omega) \ve r_x
      + (\h11 \cos\omega - \h12) \ve r_y}
   {\sqrt{\h11^2 - 2 \h11 \h12 \cos\omega + \h12^2}},
\\
D_{\eta^+} \ve n
 = \frac{\k}{\sin\omega}\,(\cos\omega\,\ve r_y - \ve r_x),
\\
D_{\xi^-} \ve n
  = \frac{\k}{\sin\omega}\,(\cos\omega\,\ve r_x - \ve r_y),
\\
D_{\eta^+} \ve n
 = \frac{\k}{\sin\omega}\,
   \frac{(\h22 - \h12 \cos\omega) \ve r_y
      + (\h22 \cos\omega - \h12) \ve r_x}
   {\sqrt{\h12^2 - 2\h12\h22 \cos\f + \h22^2}}.
\end{gather*}
Statement (i) is easily verified by checking the identities
\[
D_{\xi^+} \ve n \cdot D_{\xi^+} \ve n
 = D_{\xi^-} \ve n \cdot D_{\xi^-} \ve n
 = D_{\eta^+} \ve n \cdot D_{\eta^+} \ve n
 = D_{\eta^-} \ve n \cdot D_{\eta^-} \ve n = \k^2.
\]

Let $\psi$ denote the oriented angle $\angle(D_{\xi^-} \ve n, D_{\eta^+} \ve n)$.
To prove (ii), one easily computes
\[
\cos\psi
 = \frac{D_{\xi^-} \ve n \cdot D_{\eta^+} \ve n}{\k^2}
 = -{\cos\omega},
\]
and
\[
\sin\psi\,\ve n
 = \frac{D_{\xi^-} \ve n \times D_{\eta^+} \ve n}{\k^2}
 = -{\sin\omega}\, \ve n.
\]
Therefore,  $\psi = \pi + \omega$.
\end{proof}

Finally, it is easy to check the Lelieuvre formulas~\cite[equation~(1.140)]{R-S}
\[
D_{\xi^\pm} \ve r^\pm
 = -\frac{1}{\k} D_{\xi^\pm}\ve n \times \ve n,
\qquad
D_{\eta^\pm} \ve r^\pm
 = \frac{1}{\k} D_{\eta^\pm}\ve n \times \ve n,
\]
which relate the pseudospherical surfaces $\ve r^\pm$ to their Gauss images.

\section{From pairs of pseudospherical surfaces to concordant nets}
\label{sect:PS->ccnets}

In this section, we prove the converse of Theorem~\ref{thm:ccn->2ps}.
Given a pair of pseudospherical surfaces
of equal constant negative Gaussian curvatures,
we construct the corresponding concordant Chebyshev net.
We draw inspiration from the results of the previous section,
but the proofs have very little in common.

We denote surfaces differently from the previous section.
This is not only more convenient for the proof of Theorem~\ref{thm:2ps->ccn},
but it also helps to separate the two proofs.
The reader may wish to consult Table~\ref{tab:sect:sect} for important matches and
differences.
Note that many concepts have no counterpart in the previous section and vice versa.
\begin{table}[th]\centering\renewcommand{\arraystretch}{1.2}
\begin{tabular}{lll}
Previous section & This section \\\hline
$\ve n$ & $ \ve n$ \\
$\ve r$, $\ve r^+$, $\ve r^-$ & $\bar{\ve r}$, $\ve r$, $\ve r'$ \\
$x$, $y$, $\xi^\pm$, $\eta^\pm$ & nothing \\
nothing & $p$, $q$, $\xi$, $\eta$ \\
$D_{\xi^\pm}$, $D_{\eta^\pm}$ & $\widehat{X_i}$ for various $\epsilon_i$
\end{tabular}
\caption{Translation table between Sections~\ref{sect:ccnets->PS} and~\ref{sect:PS->ccnets}.}\label{tab:sect:sect}
\end{table}

The key idea drawn from the previous section is the parallelism
induced by the coincidence of normal vectors.

\begin{Definition}
The {\it parallelism} \cite{LPE-1909,WCG-1922,KMP-1866}
between two surfaces $S$, $S'$ is a correspondence between~$S$ and~$S'$
such that the diagram
\begin{gather} \label{eq:diag}
\begin{split}&
\xymatrix{
S \ar@{<->}[rr]^{{\rm parallelism}} \ar[rd]_{\gamma} && S' \ar[ld]^{\gamma'}\\
 & \mathbb S^2 &
}\end{split}
\end{gather}
is commutative.
Here $\mathbb S^2$ is the unit sphere, while $\gamma$, $\gamma'$
denote the Gauss maps.
\end{Definition}

Obviously by the definition of the Gauss map,
the surfaces $S$, $S'$ have equal normals and equal tangent
planes at corresponding points.
This is why the parallelism is also known as the
{\it parallelism of normals} or the
{\it parallelism of tangent planes}.

The parallelism implies the possibility to establish
local parameterisations $\ve r, \ve r' \colon  U \to \ve E^3$ that
complete the commutative diagram~\eqref{eq:diag} into
\begin{gather*}
\begin{split}&
\xymatrix{
& \ar[ld]_{\ve r} U \ar[rd]^{\ve r'}  &\\
S \ar@{<->}[rr]^{{\rm parallelism}} \ar[rd]_{\gamma} && S' \ar[ld]^{\gamma'}\\
 & \,\mathbb S^2, &
}\end{split}
\end{gather*}
whenever $\operatorname{Im} \gamma$ intersects with $\operatorname{Im} \gamma'$.
Such maps $\ve r$, $\ve r'$ will be referred to as {\it parallel
parameterisations}.
They are not unique
since they can be combined with an arbitrary diffeomorphism~${U \to U}$.

To put it simply, $\ve n = \gamma \circ \ve r = \gamma' \circ \ve r' = \ve n'$
as maps $U \to \mathbb S^2$.
For generic surfaces, $\gamma$, $\gamma'$ are local diffeomorphisms.
If this is the case, parallel parameterisations locally exist.
However, the Gauss maps need not be global diffeomorphisms
(for a wealth of beautiful examples, see~\cite{DB-2017}).

\begin{Definition}
Consider a pair of surfaces $S$, $S'$.
The locus $\bar S$ of mid-points between points related by parallelism
is called the {\it middle surface}.
\end{Definition}

More explicitly, if $\ve r(p,q)$, $\ve r'(p,q)$ are
parallel local parameterisations of surfaces $S$, $S'$,
then
\[
\bar{\ve r}(p,q)= \tfrac12\ve r(p,q) + \tfrac12\ve r'(p,q)
\]
is the parallel parameterisation od $\bar S$.
Locally, the definition does not depend on the choice of parallel
parameterisations.
Needless to say, the normals $\bar{\ve n}(p,q) = \ve n(p,q) = \ve n'(p,q)$ coincide,
showing that $\bar S$ is also related by parallelism to both $S$, $S'$.
As a case in point, the middle surface of surfaces $\ve r^\pm$ defined
by formulas~\eqref{eq:assoc} is $\ve r$ in the notation from
Section~\ref{sect:ccnets->PS}.

As is well known, every pseudospherical surface carries an
asymptotic Chebyshev net~\cite{UD-1870}.
We shall show that for a generic pair of pseudospherical surfaces
these nets combine to two concordant nets on the middle surface.
This yields the following converse of Theorem~\ref{thm:ccn->2ps}.

\begin{Theorem}\label{thm:2ps->ccn}
Consider two pseudospherical surfaces $S$, $S'$
of equal constant negative Gaussian curvatures $K = K' = -\k^2$.
Consider a parallelism between $S$ and $S'$ and the corresponding middle surface $\bar S$.
On $\bar S$, consider the images of the asymptotic lines on $S$, $S'$ under the parallelism.
Assuming that no asymptotic direction on $S$ is taken to an asymptotic direction on $S'$,
the images combine to two concordant Chebyshev nets on $\bar S$.
\end{Theorem}

Details are explained in the course of the proof.

\begin{proof}
According to Peterson~\cite[Theorem~4]{KMP-1866} or
Margulies~\cite[Theorem~4.1]{GM-1961}, we can find parameters $p$, $q$
in such a way that
\begin{gather}\label{eq:parallel}
\ve r'_p = \xi \ve r_p, \qquad
\ve r'_q = \eta \ve r_q.
\end{gather}
To make the exposition self-contained, we give necessary details of the construction
of $p$, $q$.

In an arbitrary parameterisation, we can write
\[
\ve r'_{,j} = s^i_j \ve r_{,i},
\]
where $s^i_j$ is called the mapping tensor.
In consequence of the Gauss--Weingarten equations~\eqref{eq:GW},
the compatibility conditions $\ve r'_{,ik} = \ve r'_{,ki}$ take
the form of the Margulies equations~\cite[equation~(2.6)]{GM-1961}, which is
\begin{gather}\label{eq:Marg:2.6}
s^k_{i;j} = s^k_{j;i}
\end{gather}
(semicolons denote covariant derivatives) and~\cite[equation~(2.7)]{GM-1961},
which is
\begin{gather}\label{eq:Marg:2.7}
s^k_i \II jk = s^k_j \II ik.
\end{gather}
The fundamental forms of $S$, $S'$ are related by
\begin{gather}\label{eq:Marg:ff}
\I ij' = s^k_i s^l_j \I kl,
\qquad
\II ij' = s^k_i \II kj,
\end{gather}
their determinants by
\begin{gather} \label{eq:Marg:ff:det}
\det\I{}{}' = (\det s)^2 \det\I{}{},
\qquad
\det\II{}{}' = \det s \det\II{}{},
\end{gather}
and their Gauss curvatures by
\[
K' = K/{\det s}.
\]
By assumption, $K' = K$. Therefore,
\begin{gather} \label{eq:Marg:det:1}
\det s = 1.
\end{gather}

Now, consider the eigenvalue problem for $s$ in the asymptotic parameterisation of $S$.
Then
$\II 11 = \II 22 = 0$, while $\II 12 \ne 0$, whence $s^1_1 = s^2_2$
by equation~\eqref{eq:Marg:2.7}.
If $s^1_2 s^2_1 = 0$, then either ${\II 11' = 0}$ or $\II 22' = 0$,
contrary to the assumptions.
Therefore, $s^1_2 s^2_1 \ne 0$ and $s$ has two different eigenvalues
\smash{$\xi = s^1_1 + \sqrt{s^1_2 s^2_1}$},
\smash{$\eta = s^1_1 - \sqrt{s^1_2 s^2_1}$}
(not to be confused with $\xi^\pm$, $\eta^\pm$ of the previous section).

Let $X^i$ be an eigenvector corresponding to the eigenvalue $\xi$.
The vector field $X = X^i \partial_i$ satisfies
$X \ve r' = X^j \ve r'_{,j}
 = X^j s^i_j \ve r_{,i}
 = \xi X^i \ve r_{,i}
 = \xi X \ve r$ and similarly for $Y$ and $\eta$.
The two eigenvector directions~$[X]$,~$[Y]$ are different.
Choosing parameters $p$, $q$ in such a way that $[X] = [\partial_p]$,
$[Y] = [\partial_q]$, we obtain equation~\eqref{eq:parallel}.
The mapping tensor becomes
\[
s = \begin{pmatrix} \xi & 0 \\ 0 & \eta \end{pmatrix}.
\]
Formulas~\eqref{eq:Marg:ff} read
\begin{alignat*}{4}
&\I 11' = \xi^2 \I 11,
&&\qquad \I 12' = \xi \eta \I 12,
&&\qquad \I 22' = \eta^2 \I 22,&
\\
&\II 11' = \xi \II 11,
&&\qquad \II 12' = \xi\II 12 = \eta\II 12,
&&\qquad \II 22' = \eta \II 22.&
\end{alignat*}
In particular, $\II12(\xi - \eta) = 0$.
Since $\xi\ne \eta$, we have
\[
\II 12' = \II 12 = 0.
\]
Hence,  the Peterson coordinates are conjugate on $S$ and $S'$,
which is their well-known property.

Since $\det s = 1$ by equation~\eqref{eq:Marg:det:1}, we have
\[
\eta = 1/\xi.
\]
Denoting $\Delta = \detI$, $\Delta' = \detI'$, equation~\eqref{eq:Marg:ff:det}
gives
\[
\Delta' = \Delta.
\]
By assumption, $-\kappa^2 = K = \II 11 \II 22/\Delta$.
Therefore,
\begin{gather}\label{eq:II22}
\II 22 = -\frac{\k^2}{\II 11} \Delta.
\end{gather}

Consider the middle surface $\bar{\ve r}= \frac12(\ve r + \ve r')$ now.
Using equations~\eqref{eq:parallel} with $\eta = 1/\xi$, we obtain 
\[
\ve r_p' = \frac{1 + \xi}{2} \ve r_p,
\qquad
\ve r_q' = \frac{1 + \xi}{2 \xi}  \ve r_q.
\]
For the first fundamental form, we have
\begin{gather}\label{eq:Ir}
\barI ij = \frac{(1 + \xi)^2}{4\,\xi^{i+j-2}} \I ij,
\qquad
\det\barI{}{} 
 = \frac{(1 + \xi)^4}{16\,\xi^2} \Delta.
\end{gather}
Note that the metric $\barI{}{}$ is singular at $\xi = -1$.

Since $\bar{\ve r}$, $\ve r'$, $\ve r$ have one and the same normal
vector~$\ve n$, we have
$\barII ij = \frac12(\II ij + \II ij')$, that is,
\begin{gather}\label{eq:IIij}
\barII11 = \frac{1 + \xi}{2} \II 11,
\qquad
\barII12 = 0,
\qquad
\barII22 = \frac{1 + \xi}{2\xi} \II 22
 = -\frac{1 + \xi}{2 \xi} \frac{\k^2}{\II 11} \Delta.
\end{gather}
Thus, the Gaussian curvature of $\bar{\ve r}$ is 
\begin{gather}\label{eq:K}
\bar K = \frac{\det\barII{}{}}{\det\barI{}{}}
 = -\frac{4 \k^2 \xi}{(1 + \xi)^2}.
\end{gather}
We see that the sign of $\bar K$ is that of $\xi$.
Moreover, $\xi = -1$ is a true singularity of $\bar S$.

As can be inferred from the results of the previous section,
the concordant Chebyshev net on~$\bar{\ve r}$ we look for is expected to follow
the asymptotic directions on~$\ve r$ and $\ve r'$.
Let they be represented by $X$ and $X'$, respectively.
To find the fields $X$, $X'$,
we look for functions $\zeta(p,q)$, $\zeta'(p,q)$
such that $X = D_p + \zeta D_q$, $X' = D_p + \zeta' D_q$ satisfy
$\II{}{} (X, X) = \II{}{}' (X', X') = 0$.
However, 
\begin{gather*}
\II{}{} (X, X)
 = \II11 + \zeta^2 \II22
 = \II 11 - \frac{\k^2 \zeta^2}{\II 11} \Delta, \\
\II{}{}' (X', X')
 = \II11' + \zeta^{\prime2} \II22'
 = \xi \II 11 - \frac{\k^2 \zeta^2}{\xi \II 11} \Delta,
\end{gather*}
whence 
\[
\zeta = \epsilon_1 \frac{\II 11}{\k \sqrt{\Delta}}, \qquad
\zeta' = \epsilon_2 \frac{\xi \II 11}{\k \sqrt{\Delta}},
\]
where $\epsilon_1$, $\epsilon_2$ are $\pm1$ independently.
Altogether we obtain four directions
\begin{gather*}
X_1 = D_p + \zeta D_q
 = D_p + \epsilon_1 \frac{\II 11}{\k \sqrt{\Delta}} D_q,
\\
X_2 = D_p + \zeta' D_q
 = D_p + \epsilon_2 \frac{\xi \II 11}{\k \sqrt{\Delta}} D_q.
\end{gather*}
In short, 
\[
X_i = D_p + \epsilon_i \frac{\xi^{i - 1} \II 11}{\k \sqrt{\Delta}} D_q,
\qquad i = 1,2.
\]
On $\bar{\ve r}$, the directions $[X_i]$ represent the images of
the asymptotic directions on $\ve r$, $\ve r'$ under the parallelism.
Hence, they represent the images of the asymptotic lines
mentioned in the statement of the theorem.

We shall demonstrate two ways to choose the signs $\epsilon_1$ and $\epsilon_2$
so that the net induced on $\bar{\ve r}$ is concordant Chebyshev.
In what follows, geometric objects associated with this net
are marked with tilde.

The first fundamental coefficients are
\[
\tilI ij  = \barI{}{}(X_{i},X_{j})
 = \barI11
   + \bigl(\epsilon_i \xi^{i-1}
   + \epsilon_j \xi^{j-1}\bigr) \frac{\II 11}{\k\sqrt{\Delta}} \barI12
   + \epsilon_i\epsilon_j \xi^{i+j-2} \frac{(\II 11)^2}{\k^2\Delta} \barI22,
\]
where $\barI ij$ are given by formulas~\eqref{eq:Ir}.
Hence, 
\[
\det \widetilde{{\rm I}}
  = \frac{(\epsilon_1 - \epsilon_2 \xi)^2 (\II 11)^2}{\k^2\Delta} \det\barI{}{}
  = \frac{(1 + \xi)^4 (\epsilon_1 - \epsilon_2 \xi)^2}{16 \k^2 \xi^2}
     (\II 11)^2.
\]
Likewise, the second fundamental coefficients are 
\begin{align*}
\tilII ij
& = \barII11
   + \bigl(\epsilon_i \xi^{i-1}
   + \epsilon_j \xi^{j-1}\bigr) \frac{\II 11}{\k\sqrt{\Delta}} \barII12
   + \epsilon_i\epsilon_j \xi^{i+j-2} \frac{(\II 11)^2}{\k^2\Delta} \barII22\\
 & = \frac{1 + \xi}{2} \bigl(1 - \epsilon_i\epsilon_j \xi^{i+j-3}\bigr) \II 11
\end{align*}
by virtue of formulas~\eqref{eq:IIij}.
More explicitly,{\samepage
\[
\widetilde{{\rm II}}_{11}
 =  \frac{\xi^2 - 1}{2 \xi} \II 11,
\qquad
\widetilde{{\rm II}}_{12}
 = \frac{1 + \xi}{2} (1 - \epsilon_1\epsilon_2) \II 11,
\qquad
\widetilde{{\rm II}}_{22}
 = \frac{1 - \xi^2}{2} \II 11.
\]
If $\epsilon_1 = \epsilon_2$, then
$\sigma = \widetilde{{\rm II}}_{12}/\!\sqrt{\det \widetilde{{\rm I}}} = 0$,
which rules out the concordant net.}

Continuing with $\epsilon_1 \ne \epsilon_2$, we get
\begin{gather*}
\det \widetilde{{\rm I}}
 = (1 + \xi)^4\,
   \frac{(\epsilon_1 - \epsilon_2 \xi)^2}{16 \k^2 \xi^2} (\II11)^2,
\qquad
\sqrt{\det \widetilde{{\rm I}}}
 = \frac{(1 + \xi)^2}{4}\,
   \left|\frac{\epsilon_1 - \epsilon_2 \xi}{\xi \k}\,\II11\right|,
\\
\widetilde\sigma = \frac{\widetilde{{\rm II}}_{12}}{\sqrt{\det \widetilde{{\rm I}}}}
 = \pm\frac{4 \k \xi}
        {(1 + \xi)^2}
\end{gather*}
according to equation~\eqref{eq:sigma}.
The sign $\pm$ depends on whether $\II11 \gtrless 0$, $\xi \gtrless 0$ and
$\epsilon_1 - \epsilon_2 \xi \gtrless 0$,
being undefined at the singularity $\xi = -1$.
Anyway, we have
\[
\widetilde K \pm \kappa \widetilde\sigma = 0
\]
by comparison with equation~\eqref{eq:K} (obviously, $\widetilde K = \bar K$).
Consequently, we obtain two concordant nets,
one for $\epsilon_1 = 1$, $\epsilon_2 = -1$,
the other one for $\epsilon_1 = -1$, $\epsilon_2 = 1$.
Note also that the sign of~$\widetilde\sigma$ is changeable by more than one
discrete symmetry, see Table~\ref{tab:Ts:inv}.

It remains to be proved that the net has the Chebyshev property,
which can be done by proving that $\widetilde\pi_1 = \widetilde\pi_2 = 0$ or,
equivalently, that $\widetilde\Gamma^1_{12} = \widetilde\Gamma^2_{21} = 0$.
It is a matter of direct verification that the values computed
according to equation~\eqref{eq:Gamma} are zero modulo
certain valid identities we list in the sequel.

Denoting by $\Gamma^i_{jk}(p,q)$ the Christoffel symbols with respect
to the Levi-Civita connection for the metric $\I{}{}$, and by a semicolon the
corresponding covariant derivatives,
the Mainardi--Codazzi equations
$\mathrm{II}_{ij;k} - \mathrm{II}_{ik;j} = 0$
for $\ve r$, cf.~equation~\eqref{eq:GMC}, reduce to
\begin{gather}\label{eq:MC}
\begin{split}
& \mathrm{MC}_1 \equiv \frac{\partial\II 11}{\partial q}
 - \II 11 \Gamma^1_{12} + \II 22 \Gamma^2_{11} = 0,
\\
& \mathrm{MC}_2 \equiv \frac{\partial\II 22}{\partial p}
 + \II 11 \Gamma^1_{22} - \II 22 \Gamma^2_{12} = 0,
 \end{split}
\end{gather}
where $\II 22$ is to be substituted from equation~\eqref{eq:II22}.

The Margulies equations~\eqref{eq:Marg:2.6} reduce to
\begin{gather}\label{eq:Marg}
\begin{split}
& \mathrm{Marg}_1 \equiv \frac{\partial \xi}{\partial q}
 - \frac{1 - \xi^2}{\xi} \Gamma^1_{12} = 0,
\\
& \mathrm{Marg}_2 \equiv -\frac{1}{\xi^2} \frac{\partial \xi}{\partial q}
 + \frac{1 - \xi^2}{\xi} \Gamma^2_{12} = 0.
 \end{split}
\end{gather}
Now it is straightforward to check that
\begin{align*}
\widetilde\Gamma^1_{12} ={}&
 -\frac{\xi}{1 + \xi} \left(
    \frac{\epsilon_1}{\sqrt\Delta} \mathrm{MC}_1 
   + \frac{\II11}{\kappa\Delta} \mathrm{MC}_2 
   + \frac{\epsilon_1^2 - 1}{\epsilon_1}\, 
     \frac{\kappa^2 \sqrt\Delta}{\II11} \Gamma^2_{11}
   + \bigl(\epsilon_1^2 - 1\bigr) \frac{(\II11)^2}{\kappa\Delta} \Gamma^1_{22} 
 \right),
\\
\widetilde\Gamma^2_{21}   ={}&
 \frac{\xi}{1 + \xi} \left(
   \frac{\epsilon_1}{\sqrt\Delta} \mathrm{MC}_1 
    - \k\,\mathrm{Marg}_2 
    + \frac{\epsilon_1^2 - 1}{\epsilon_1}
      \frac{\kappa^2 \sqrt\Delta}{\II11} \Gamma^2_{11} 
 \right)
\\
 & - \frac{1}{1 + \xi} \left(
    \frac{\II11}{\k \Delta} \mathrm{MC}_2 
     + \epsilon_1 \frac{\II11}{\sqrt\Delta} \mathrm{Marg}_1 
     + \frac{\epsilon_1^2 - 1}{\epsilon_1}\,
        \kappa\Delta (\II11)^2 \Gamma^1_{22} \right)
\end{align*}
vanish in consequence of equations~\eqref{eq:MC} and \eqref{eq:Marg} and
$\epsilon_1 = \pm1$.
This finishes the proof of Theorem~\ref{thm:2ps->ccn}.
\end{proof}

Theorem~\ref{thm:2ps->ccn} provides a geometric solution to problem (B).
In principle, this geometric solution can be turned into an analytic
solution of system~\eqref{GMC} and~\eqref{eq:h22} in implicit form,
but the result is too complex to be of any use.

It is worth mentioning that this construction yields
Chebyshev nets, but not Chebyshev parameterisations in the sense of
Proposition~\ref{prop:Chebyshev}\,(i), which underlines the
importance of distinguishing between the two concepts.

\begin{Corollary}
The class of surfaces admitting a concordant Chebyshev net
coincides with the class of middle surfaces of pairs of
pseudospherical surfaces under the correspondence by equal normals.
\end{Corollary}

At the end of Section~\ref{sect:ccnets->PS}, we observed that
every concordant net induces a pair of Chebyshev nets on the
unit sphere; the explicit description was given in
Proposition~\ref{prop:Chpair}.
The following proposition provides a version of Theorem~\ref{thm:2ps->ccn} starting with two
Chebyshev nets on the sphere.

\begin{Corollary}
Consider the unit sphere $\| \ve n \| = 1$ carrying two Chebyshev nets given by
directions $[X^\pm_1]$, $[X^\pm_2]$, where $(X^\pm_1, X^\pm_2)$ are two pairs
of commuting unit vector fields.
Then we can choose the signs in such a way that both $X^+_1$, $X^-_2$ and $X^-_1$, $X^+_2$
represent concordant Chebyshev nets on the surface
$\ve r = \frac12 \ve r^+ + \frac12 \ve r^-$, where  surfaces $\ve r^\pm$ are determined
by the Lelieuvre formulas
\[
X_1^\pm \ve r^\pm
 = -\frac{1}{\k} X_1^\pm \ve n \times \ve n,
\qquad
X_2^\pm \ve r^\pm
 = \frac{1}{\k} X_2^\pm \ve n \times \ve n
\]
and correspond by the parallelism of normals.
\end{Corollary}

\begin{proof}
Obvious.
Note that $\ve r^+$, $\ve r^-$, $\ve r$ correspond to $\ve r$, $\ve r'$, $\bar{\ve r}$, respectively.
\end{proof}

\section{Examples}
\label{sect:ex}

In this section, we discuss explicit examples based on
Theorem~\ref{thm:2ps->ccn}.
We switch back to the notation of Section~\ref{sect:ccnets->PS},
cf.\ Table~\ref{tab:sect:sect}.
In particular, $\ve r^+$, $\ve r^-$, $\ve r$ of this section are $\ve r$, $\ve r'$, $\bar{\ve r}$
of Section~\ref{sect:PS->ccnets}.
For the reader's convenience, we review the construction.

\begin{construction}
The input is a pair of pseudospherical surfaces $P^+$ and $P^-$.
\begin{enumerate}\itemsep=0pt\samepage
\item Relate $P^+$ and $P^-$ by parallelism, i.e., choose parameters $p$, $q$ so that
$\ve n^{+}(p,q) = \ve n^{-}(p,q)$.

\item Compute the middle surface
$\ve r(p,q) = \frac12 \ve r^+(p,q) + \frac12 \ve r^-(p,q)$.

\item Find the asymptotic lines on $P^+$ and $P^-$, altogether four line families.

\item Find the corresponding four line families on the middle surface.

\item Select the two pairs that form the two concordant Chebyshev nets sought.
\end{enumerate}
\end{construction}

\begin{Example}
Consider two pseudospheres $\ve r^{+}$ and $\ve r^{-}$ with
perpendicular axes parallel to the $x$- and $y$-axis, respectively.
In isodiagonal parameterisations, see Remark~\ref{rem:bisect}, we have
\begin{gather} \label{eq:psps:rr}
\begin{split}
& \ve r^{+} =
\left[v^+ - \tanh v^+,
\frac{\cos u^+}{\cosh v^+},
\frac{\sin u^+}{\cosh v^+}\right],
\\
& \ve r^{-} =
\left[\frac{\cos u^-}{\cosh v^-},
v^- - \tanh v^-,
\frac{\sin u^-}{\cosh v^-}\right],
\end{split}
\end{gather}
assuming $u^\pm \in \mathbb S^1$ and $v^\pm \in \mathbb R$.

The Gauss maps are almost bijective if using the outward (or inward) normals.
Figure~\ref{fig:rainbow} is coloured in such a way that the
Gauss mapping of the pseudosphere
(which is also a parallelism between the pseudosphere and the sphere)
is colour-preserving.
\begin{figure}[ht]\centering

\includegraphics[scale = 0.28, angle = 90]{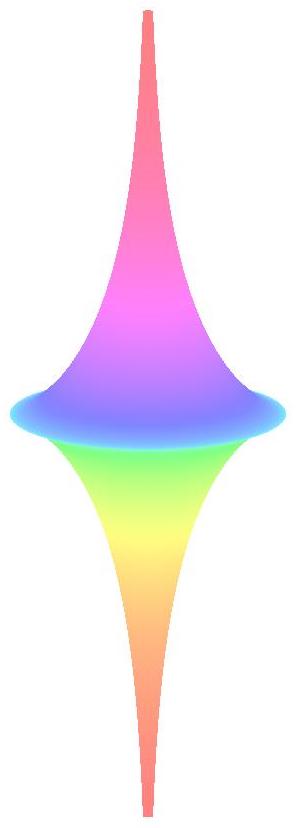}
\qquad
  \includegraphics[scale = 0.14, angle = 90]{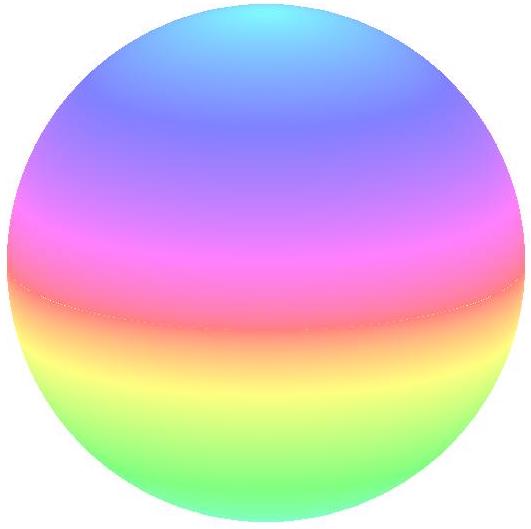}

\caption{Colour visualisation of the Gauss map by outward normals.}
\label{fig:rainbow}
\end{figure}

The coordinate formulas are
\begin{gather} \label{eq:psps:nn}
\begin{split}
& \ve n^{+} = \sign v^+
\left[\frac{1}{\cosh v^+},
\tanh v^+ \cos u^+,
\tanh v^+ \sin u^+\right],
\\
& \ve n^{-} =
\sign v^- \left[\tanh v^- \cos u^-,
\frac{1}{\cosh v^-},
\tanh v^- \sin u^-\right],
\end{split}
\end{gather}
where $\sign v^\pm$ ensure that the normals are outward.

To perform Step 1, we relate parameters $u^\pm$, $v^\pm$
by $\ve n^{+}(u^+,v^+) = \ve n^{-}(u^-,v^-)$.
This can be done in various ways.
Denoting by $R^\pm_i$ and $N_i$ the components of
$\ve r^\pm$ and $\ve n = \ve n^\pm$, respectively, the inverse Gauss maps
$(\gamma^\pm)^{-1}$ are
\begin{gather*}
R^+_1 = \sign N_1 \left(\operatorname{arcosh} \left|\frac1{N_1}\right| - \sqrt{1 - N_1^2}\right), \qquad
R^+_i = \frac{|N_1|\, N_i}{\sqrt{1 - N_1^2}}, \qquad i = 2,3,
\\
R^-_2 = \sign N_2 \left(\operatorname{arcosh}\left|\frac1{N_2}\right| - \sqrt{1 - N_2^2}\right), \qquad
R^-_i = \frac{N_i \,|N_2|}{\sqrt{1 - N_2^2}}, \qquad i = 1,3,
\end{gather*}
assuming $N_1^2 + N_2^2 + N_3^2 = 1$.
Substituting
\begin{gather*}
N_1 = \cos\phi \cos\theta, \qquad
N_2 = \sin\phi \cos\theta, \qquad
N_3 = \sin\theta, \qquad
-\tfrac12\pi < \theta < \tfrac12\pi,  \quad
-\pi < \phi < \pi,
\end{gather*}
we get $\ve r^\pm(\phi,\theta)$ in spherical coordinates on the Gauss sphere.
Thus,{\samepage
\begin{align*}
\ve r^+(\phi,\theta) = \biggl[ &\pm
\operatorname{arcosh}\left|\frac1{\cos\phi\,\cos\theta}\right|
 \mp \sqrt{1 - \cos^2\phi\,\cos^2\theta},
\frac{\left|\cos\phi\right|\sin\phi\,\cos^2\theta}
  {\sqrt{1 - \cos^2\phi\,\cos^2\theta}},
\\
& \frac{\left|\cos\phi\right|\sin\theta\,\cos\theta}
  {\sqrt{1 - \cos^2\phi\,\cos^2\theta}}
\biggr],
\\
\ve r^-(\phi,\theta) = \biggl[& \frac{\left|\sin\phi\right|\cos\phi\cos^2\theta}
  {\sqrt{1 - \sin^2\phi\,\cos^2\theta}},
\pm\operatorname{arcosh}\left|\frac1{\cos\phi\,\cos\theta}\right|
 \mp \sqrt{1 - \sin^2\phi\,\cos^2\theta},
\\
& \frac{\left|\sin\phi\right|\sin\theta\,\cos\theta}
  {\sqrt{1 - \sin^2\phi\,\cos^2\theta}}
\biggr],
\end{align*}
where $\pm = \sign (\cos\phi)$.}

To perform Step 2, we compute
\begin{gather}\label{eq:2ps:mid}
\ve r(\phi,\theta) = \tfrac12 \ve r^+(\phi,\theta) + \tfrac12 \ve r^-(\phi,\theta).
\end{gather}
This is the middle surface, a snippet of which is displayed in Figure~\ref{fig:mid}
(blue for $0 < \theta < \frac12 \pi$,
yellow for $-\frac12 \pi < \theta < 0$), $0 < \phi < \frac12 \pi$,
restricted to $x < 2$, $y < 2$.
The whole middle surface has four connected components, obtainable by
rotating one of them by $\frac12 \pi$, $\pi$, $\frac32 \pi$ around the $z$-axis.
All parts extend to infinity along the $x$- and $y$-axis
(here $x$, $y$, $z$ refer to coordinates in Euclidean space).
\begin{figure}[ht]
\centering
\includegraphics[scale = 0.33]{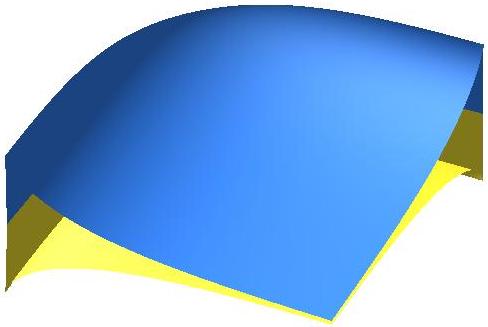}
\caption{A snippet of the middle surface of two pseudospheres.}\label{fig:mid}
\end{figure}

The middle surface is regular except eight cuspidal edges,
two of which are clearly seen in Figure~\ref{fig:mid}.
Their Gauss image consists of four adjacent ovals,
formed by zeroes of certain
polynomial $\Pi(\cos\phi, \cos\theta)$, which is too large to be printed.\footnote{The ovals $\Pi(\cos\phi, \cos\theta) = 0$
are miraculously well approximated
by the ellipses
$\phi = \frac14 \pi(2 k - 1 + \cos t)$, $k = 1,\dots,4$,
$\theta = \arccos\sqrt{2 - \sqrt 2} \cdot \sin t$
in the $\phi,\theta$-plane.}
The Gauss images of cuspidal edges are drawn in white in
Figure~\ref{fig:rainbow:sing}
(blue hemisphere for $\theta > 0$, yellow for $\theta < 0$).
The Gauss curvature of $R(\phi,\theta)$ is
negative for $\phi$, $\theta$ inside the ovals and
positive for $\phi$, $\theta$ outside the ovals
(compare Figures~\ref{fig:mid} and~\ref{fig:rainbow:sing}).

\begin{figure}[ht]
\centering
  \includegraphics[scale = 0.17]{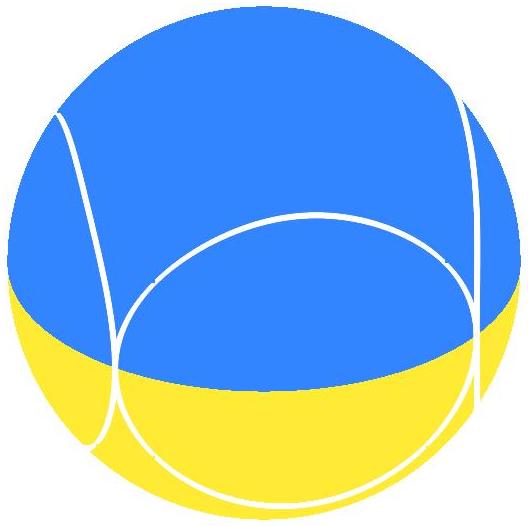}
\caption{Gaussian images of the cuspidal edges.}
\label{fig:rainbow:sing}
\end{figure}

Summarising, points~\eqref{eq:2ps:mid}
fill the middle surface and are regular
if $\Pi(\cos\phi, \cos\theta) \ne 0$.
Figure~\ref{fig:triples} visualises the middle points
for $\phi$, $\theta$ in different positions relative to the ovals.
From left to right, the curvature in $R(\phi,\theta)$ is
negative, singular (cuspidal edge) and positive, respectively.
The colours indicate individual surfaces
(pseudospheres are red and blue, the middle surface is yellow).
Short sticks represent outward normals.
\begin{figure}[ht]
\centering
  \includegraphics[scale = 0.282]{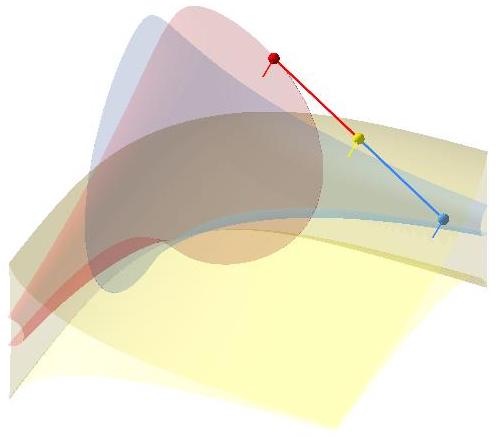}
\quad
  \includegraphics[scale = 0.282]{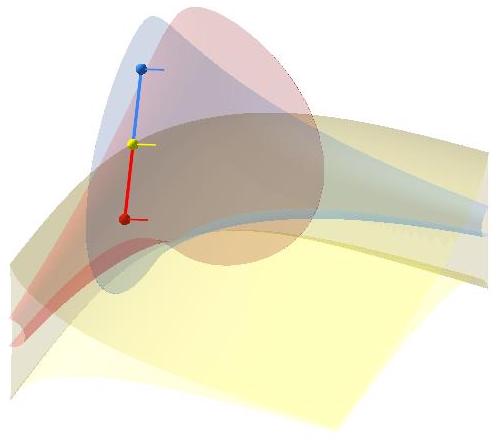}
\quad
  \includegraphics[scale = 0.282]{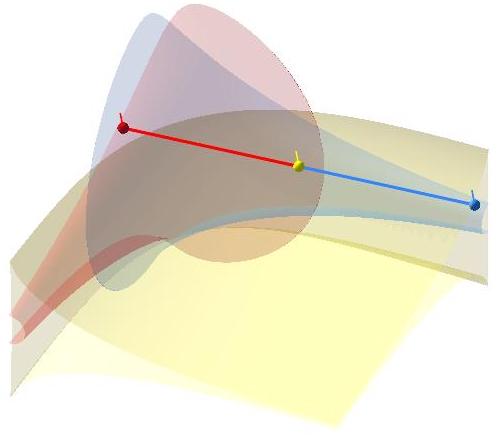}
\caption{Various positions of $R(\phi,\theta)$.}
\label{fig:triples}
\end{figure}

In Step 3, we equip the two pseudospheres with their asymptotic
Chebyshev parameterisations~$x^\pm$,~$y^\pm$.
These can be found by substituting~$u^\pm = x^\pm + y^\pm$,~$v^\pm = x^\pm - y^\pm$
into~\eqref{eq:psps:rr} since~$u^\pm$,~$v^\pm$ are isogonal on the pseudospheres~$\ve r^\pm$.
We get
\begin{gather*}
\ve r^{+}(x^+, y^+) =
\left[x^+ - y^+ - \tanh (x^+ - y^+),
\frac{\cos (x^+ + y^+)}{\cosh (x^+ - y^+)},
\frac{\sin (x^+ + y^+)}{\cosh (x^+ - y^+)}\right],
\\
\ve r^{-}(x^-, y^-) =
\left[\frac{\cos (x^- + y^-)}{\cosh (x^- - y^-)},
x^- - y^- - \tanh (x^- - y^-),
\frac{\sin (x^- + y^-)}{\cosh (x^- - y^-)}\right].
\end{gather*}
Figure~\ref{fig:ps:net} shows the result.
\begin{figure}[ht]
\centering
  \includegraphics[scale = 0.4, angle = 0]{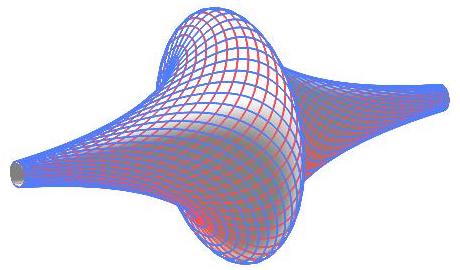}
\quad
  \includegraphics[scale = 0.4, angle = 0]{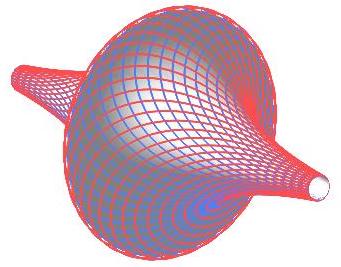}
\caption{Asymptotic Chebyshev nets on the parent pseudospheres.}
\label{fig:ps:net}
\end{figure}

In Step 4, we construct the corresponding lines on the middle surface.
We first substitute
$u^\pm = x^\pm + y^\pm$, $v^\pm = x^\pm - y^\pm$
into~\eqref{eq:psps:nn} to get the corresponding Chebyshev nets on the
Gaussian spheres, obtaining
\begin{gather*} 
\ve n^{+} =
\left[\frac{1}{\cosh (x^+ - y^+)},
\tanh (x^+ - y^+) \cos (x^+ + y^+),
\tanh (x^+ - y^+) \sin (x^+ + y^+)\right],
\\
\ve n^{-} =
\left[\tanh (x^- - y^-) \cos (x^- + y^-),
\frac{1}{\cosh (x^- - y^-)},
\tanh (x^- - y^-) \sin (x^- + y^-)\right].
\end{gather*}
Denoting by $N_1$, $N_2$, $N_3$ individual components of vectors
$\ve n^{+}(x^+,y^+)$ and $\ve n^{-}(x^-,y^-)$, the map
\begin{align*}
R = \frac12
\biggl[ &\frac{N_1 N_2}{\sqrt{1 - N_2^2}} + \operatorname{arcosh}\left(\frac1{N_1}\right) - \sqrt{1 - N_1^2},
\\
& \frac{N_1 N_2}{\sqrt{1 - N_1^2}} + \operatorname{arcosh}\left(\frac1{N_2}\right) - \sqrt{1 - N_2^2},
\frac{N_1 N_3}{\sqrt{1 - N_1^2}} + \frac{N_2 N_3}{\sqrt{1 - N_2^2}}
\biggr]
\end{align*}
allows us to obtain explicitly four line families on the middle surface.

In Step 5, we choose appropriate pairs that are guaranteed to form concordant
Chebyshev nets by Theorem~\ref{thm:2ps->ccn}.
Figure~\ref{fig:mid:net} shows the results in the straight and overturned view.
Thus, the resulting nets are composed of curves
$x^\pm = {\rm const}$ and $y^\pm = {\rm const}$
corresponding to equally coloured asymptotic curves in Figure~\ref{fig:ps:net}.
They approximate a Chebyshev parameterisation quite well,
but actually they only satisfy the curvilinear parallelogram condition,
see Section~\ref{sect:Chebyshev}.
The two nets are different, but identifiable by the mirror symmetry.

\begin{figure}[ht]\centering
  \includegraphics[scale = 0.2]{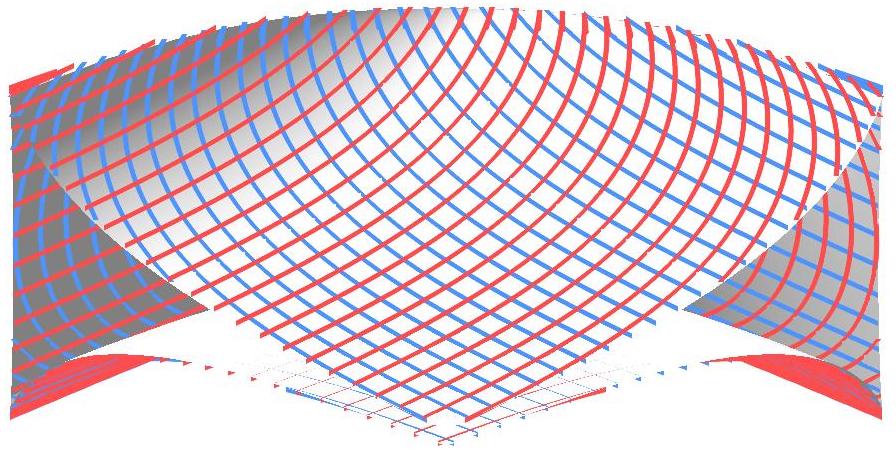}
\quad
  \includegraphics[scale = 0.2]{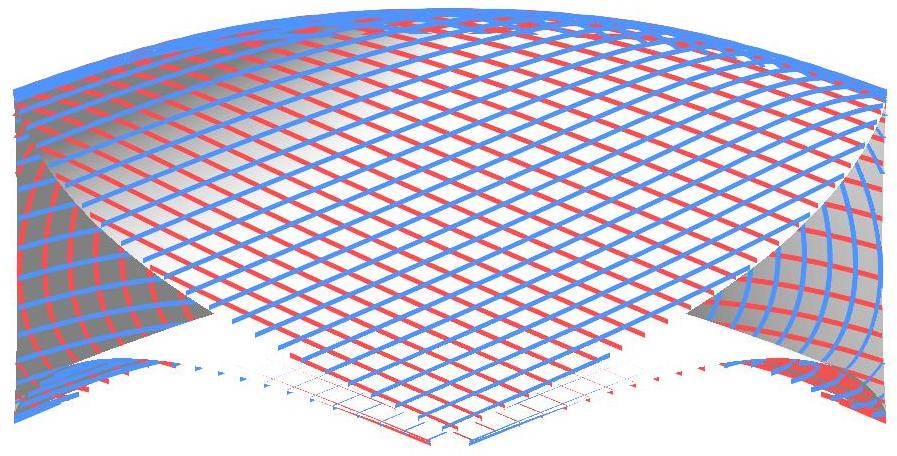}

  \vspace{2mm}

  \includegraphics[scale = 0.2]{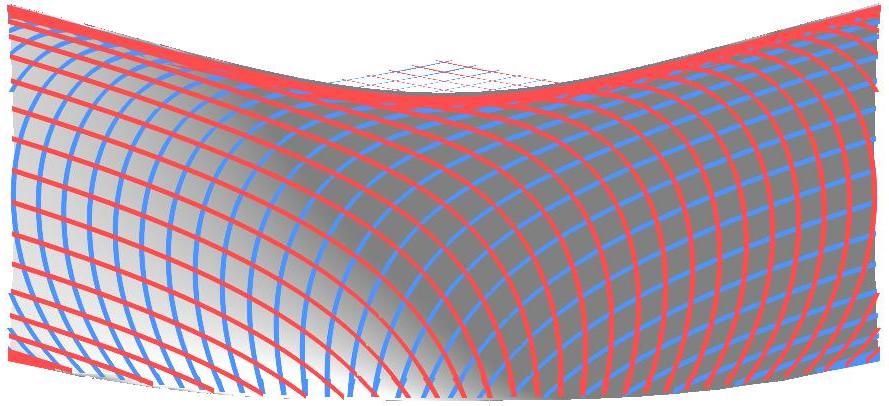}
\quad
  \includegraphics[scale = 0.2]{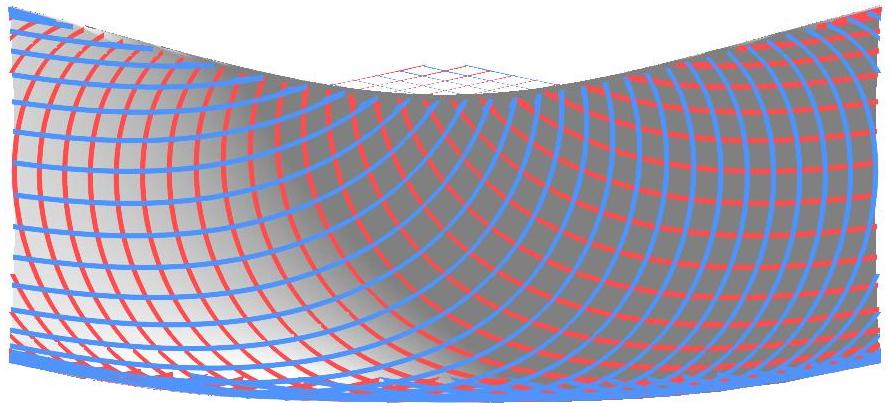}
\caption{Different concordant Chebyshev nets on the middle surface.}
\label{fig:mid:net}
\end{figure}
\end{Example}

\begin{Example}
\label{ex:hat}
Here we choose $\ve r^+$ to be the pseudosphere and $\ve r^-$ to be
one period of a coaxial pseudospherical surface of revolution of elliptic
type~\cite[Section~103]{LB-I}.
Positioning the common axis in the $z$-direction, we can write
\begin{gather*}
\ve r^{+} =
\biggl[\frac{\cos u^+}{\cosh v^+},
 \frac{\sin u^+}{\cosh v^+},
 v^+ - \tanh v^+\biggr],
\\
\ve r^{-} =\biggl[\sn\bigl(v^- \cos k \mid -{\tan^2 k}\bigr) \cos u^-,
 \sn\bigl(v^- \cos k \mid -{\tan^2 k}\bigr) \sin u^-,
\\
\hphantom{\ve r^{-} =\biggl[}
\frac{v^- - \E\bigl(\sn\bigl(v^- \cos k \mid -{\tan^2 k}\bigr) \mid -{\tan^2 k}\bigr) \cos k}
     {\sin k}\biggr]
\end{gather*}
in the isodiagonal parameterisation.
Here $\sn$ is the elliptic sine and $E$ is the elliptic integral of the
second kind, i.e.,
\[
\sn(\phi | m) = \sin \am(\phi | m),
\qquad
\E (s | m) = \int_0^s \sqrt{1 - m \sin^2 t} \,{\rm d}t.
\]
The elliptic amplitude $\am(\phi | m)$ is the inverse of the elliptic
integral of the first kind, that is, the value $s$ such that{\samepage
\[
\phi = \F (s | m) = \int_0^s \frac{{\rm d}t}{\sqrt{1 - m \sin^2 t}}.
\]
While $u^\pm \in \mathbb S^1$, the range of $v^\pm$ will be determined later.}

If using the outward normals, the Gauss image of the latter consists
of two spherical caps, see Figure~\ref{fig:rainbow_cap}. In particular, the Gauss map is not surjective.

\begin{figure}[ht]
\centering
\includegraphics[scale = 0.24, angle = 90]{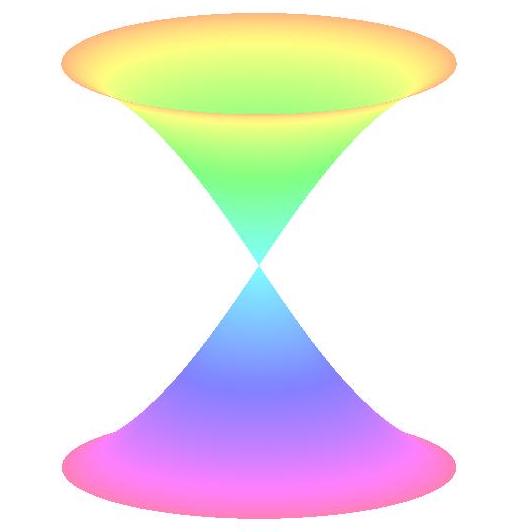}
\qquad
  \includegraphics[scale = 0.40, angle = 90]{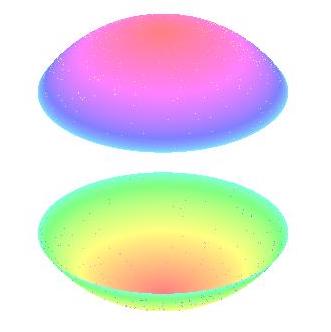}
\caption{Colour visualisation of the Gauss map in Example~\ref{ex:hat}.}
\label{fig:rainbow_cap}
\end{figure}

To perform Step~1, we need formulas for the unit normals
(the Gauss maps to $\mathbb S^2$), which are
\begin{gather*}
\ve n^{+} =
\biggl[
\tanh v^+ \cos u^+,
\tanh v^+ \sin u^+,
\frac{1}{\cosh v^+}\biggr],
\\
\ve n^{-} =
-\bigl[\sin k \cn\bigl(v^- \cos k \mid -{\tan^2 k}\bigr) \cos u^-,
  \sin k \cn\bigl(v^- \cos k \mid -{\tan^2 k}\bigr) \sin u^-,
\\
\hphantom{\ve n^{-} =-\bigl[}
  \cos k \dn \bigl(v^- \cos k \mid -{\tan^2 k}\bigr)
\bigr],
\end{gather*}
where $\dn(x | m) = \partial \am(x | m)/\partial x$.
The normals point outwards if $v^+ \in [-\operatorname{artanh}(\sin k), 0]$ and
$v^- \in [0, K(\sin^2 k)]$,
where $K(m) = F(1|m)$ is the complete elliptic integral of the first kind.
This choice covers the downward pointing cap of the elliptic pseudospherical
surface of revolution and a nozzle-shaped section of the downward pointing half
of the pseudosphere if the $z$-axis is considered vertical,
see the two outer surfaces in Figure~\ref{fig:threecoaxial}.

\begin{figure}[ht]\centering
\includegraphics[scale = 0.44, angle = 0]{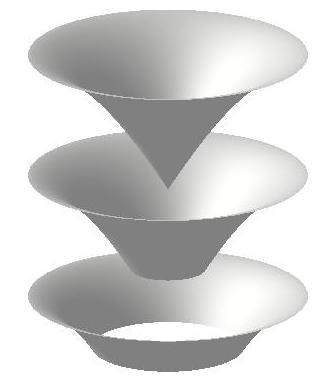}
\caption{From top to down, the three coaxial surfaces
$\ve r^-$, $\ve r$, $\ve r^+$ (separated for visibility).}
\label{fig:threecoaxial}
\end{figure}

To obtain the parallelism, we consider the equality
$\ve n^{+}(u^+,v^+) = \ve n^{-}(u^-,v^-)$, which reduces~to
\begin{gather}\label{eq:ex2:paral:impl}
u^+ = u^-,
\qquad
\sin k \cn\bigl(v^- \cos k \mid -{\tan^2 k}\bigr) + \tanh v^+ = 0.
\end{gather}
The latter equation can be solved for $v^+$ or $v^-$, giving either
\[
u^- = u^+, \qquad
v^- = \frac1{\cos k} \arccn\left(-\frac{\tanh v^+}{\sin k} \,\biggl|\, -{\tan^2 k}\right),
\]
where $v^+ \in [-\operatorname{artanh}(\sin k), 0]$,
or
\[
u^+ = u^-, \qquad
v^+ = \operatorname{arcosh}\left(\frac{1}{\dn \bigl(v^- \cos k \mid -{\tan^2 k}\bigr) \cos k}\right),
\]
where $v^- \in \bigl[0, K\bigl(\sin^2 k\bigr)\bigr]$.
With the help of these we can switch from the parameterisation by~$u^+$,~$v^+$ to the parameterisation by~$u^-$,~$v^-$ and vice versa.

In Step 2, we compute the middle surface.
We  display only the picture, see Figure~\ref{fig:threecoaxial},
suppressing the complicated formulas.

The Gauss curvature of the middle surface is
$-2 \sin^2 k/\bigl(1 + \sin^2 k\bigr)$ at the rim $v^+ = 0$
and tends to zero at the aperture $v^+ = -\operatorname{artanh}(\sin k)$.
Thus, although hyperbolic, the middle surface is not pseudospherical.

In Step 3, we have to find the asymptotic Chebyshev parameterisations of
the initial surfaces~$\ve r^+$,~$\ve r^-$.
As in the previous example, we only have to substitute
$u^\pm = x^\pm + y^\pm$, $v^\pm = x^\pm - y^\pm$ into the above formulas
for $\ve r^+(u^+,v^+)$, $\ve r^-(u^-,v^-)$.
The asymptotic Chebyshev net on $\ve r^+$ has been visualised above in
Figure~\ref {fig:ps:net},
for $\ve r^-$ see Figure~\ref{fig:pshatnet}.
\begin{figure}[ht]\centering
\includegraphics[scale = 0.065]{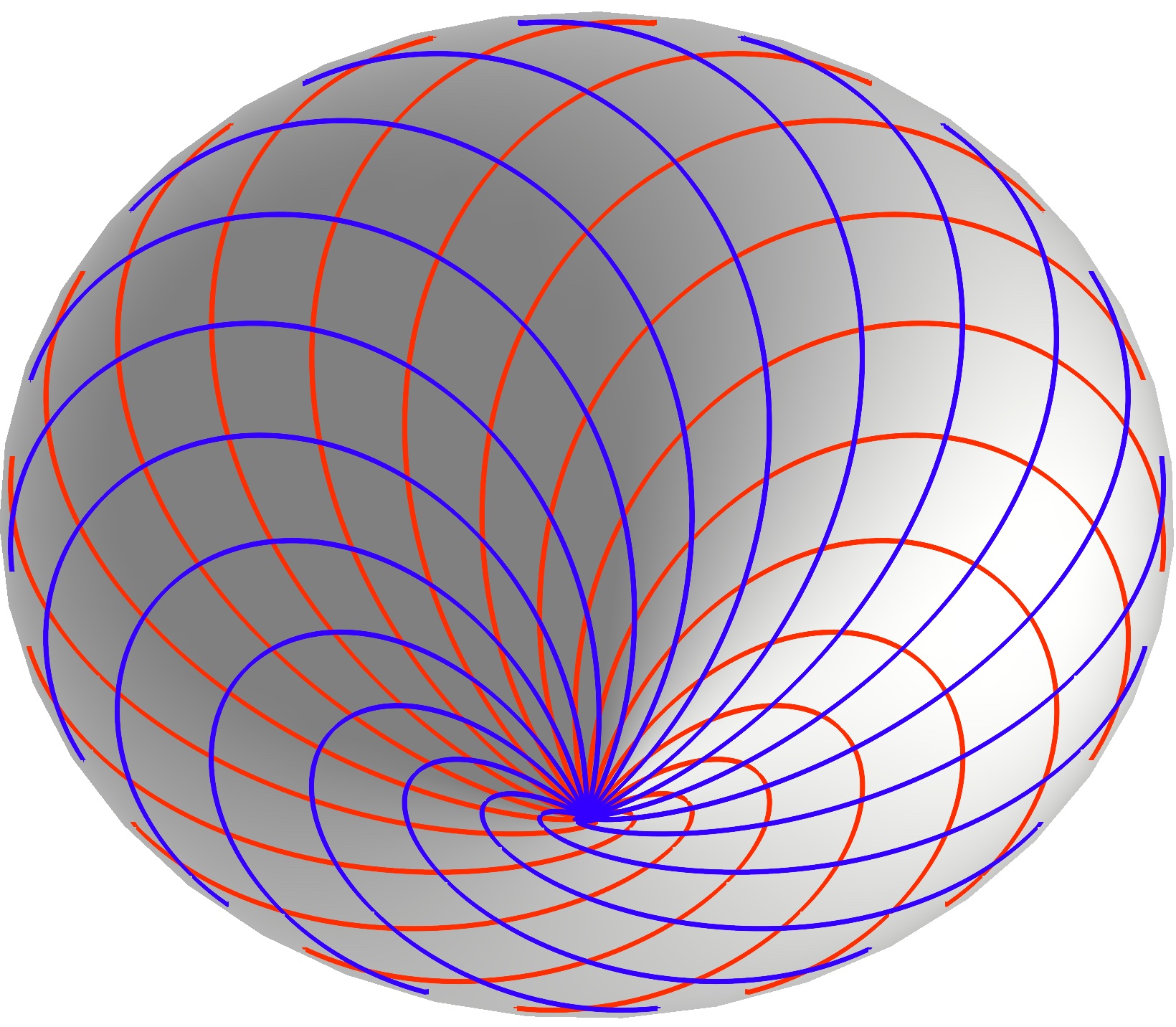}
\caption{The asymptotic Chebyshev net on $\ve r^-$.}
\label{fig:pshatnet}
\end{figure}

To perform Step 4 and find the corresponding nets on the middle surface,
we proceed differently from the previous example.
In order to be able to write formulas,
although only in principle and not fully explicit,
we express $x^+$, $y^-$ in terms of $x^-$, $y^+$.
Eliminating $u^\pm$, $v^\pm$ from
\[
x^\pm + y^\pm = u^+ = u^-, \qquad
x^\pm - y^\pm = v^\pm
\]
and equation~\eqref{eq:ex2:paral:impl},
we get
\begin{gather}
\label{eq:ex2:para:xy}
\begin{split}
&x^+ + y^+ = x^- + y^-, \\
&\sin k \cn\bigl((x^- - y^-) \cos k \mid -{\tan^2 k}\bigr)
 + \tanh (x^- + y^+) = 0.
\end{split}
\end{gather}
Denoting
$w = x^- - y^+$, $v = v^+ = x^+ - y^+$,
we substitute
\[
x^+ = v + y^+, \qquad
x^- = w + y^+.
\]
into equations~\eqref{eq:ex2:para:xy} to get
\[
y^- = y^+ + v - w, \qquad
\sin k \cn\bigl((2 w - v) \cos k \mid -{\tan^2 k}\bigr) + \tanh v = 0.
\]
From the latter equation, we can express $w$ as a function of
$v$, namely
\[
w = \Psi_k(v) = \frac{v}{2} + \frac{1}{2 \cos k}
 \operatorname{arccn}\left(-\frac{\tanh v}{\sin k}
       \,\biggl|\, -{\tan^2 k}\right) .
\]
This opens the way to express $v$ as
$\Psi_k^{-1}(w)$ and compute it at least numerically.
For the graphs, see Figure~\ref{fig:Psi}.
\begin{figure}[th]\centering
\includegraphics[scale = 0.37, angle = 0]{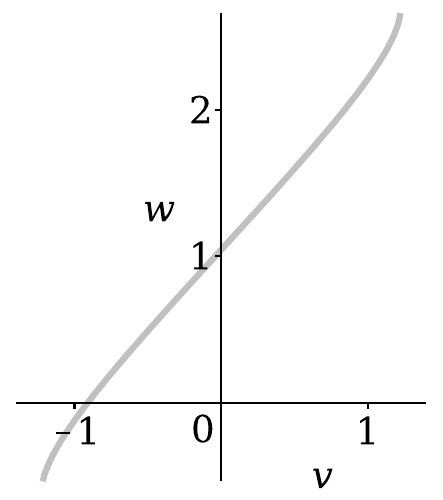}
\qquad
\includegraphics[scale = 0.44, angle = 0]{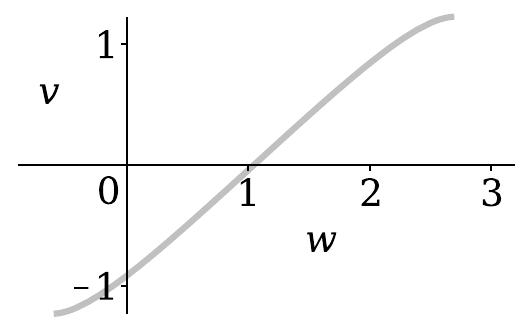}
\caption{The graphs of $w = \Psi_k(v)$ and $v = \Psi_k^{-1}(w)$ for $k = 1$.}\label{fig:Psi}
\end{figure}

The derivatives are
\[
\frac{\dif\Psi_k(v)}{\dif v} = \frac
   {1 + \sqrt{1 - \cos^2 k \cosh^2 v}}
   {2  \sqrt{1 - \cos^2 k \cosh^2 v}},
\qquad
\frac{\dif\Psi_k^{-1}(w)}{\dif w} = \frac
   {2  \sqrt{1 - \cos^2 k \cosh^2 \smash{\Psi_k^{-1}(w)}}}
   {1 + \sqrt{1 - \cos^2 k \cosh^2 \smash{\Psi_k^{-1}(w)}}}.
\]

Summarising, the resulting expressions for $x^+$, $y^-$ in terms of $x^-$, $y^+$ are{\samepage
\[
x^+ = y^+ + \Psi_k^{-1}(x^- - y^+), \qquad
y^- = 2 y^+ - x^- + \Psi_k^{-1}(x^- - y^+).
\]
These allow us to obtain parallel parameterisations
$\ve r^+(x^-,y^+)$ and $\ve r^-(x^-,y^+)$.}

By symmetry, we can also write $x^-$ and $y^+$ in terms of $x^+$ and $y^-$
and obtain parallel parameterisations
$\ve r^+(x^+,y^-)$ and $\ve r^-(x^+,y^-)$.

Step 5.
The resulting concordant Chebyshev nets are
\begin{gather*}
\ve r(x^-,y^+) = \tfrac12 \ve r^+(x^-,y^+) + \tfrac12 \ve r^-(x^-,y^+),
\\
\ve r(x^+,y^-) = \tfrac12 \ve r^+(x^+,y^-) + \tfrac12 \ve r^-(x^+,y^-).
\end{gather*}
For the plots see Figure~\ref{fig:twocoaxial}.
Again, the two nets are different, but identifiable by the mirror symmetry.
\begin{figure}[th]\centering
\includegraphics[scale = 0.35]{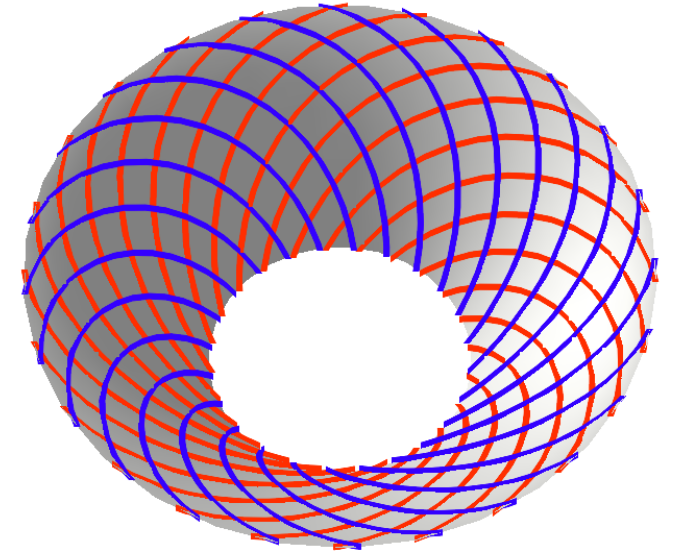}
\qquad
\includegraphics[scale = 0.35]{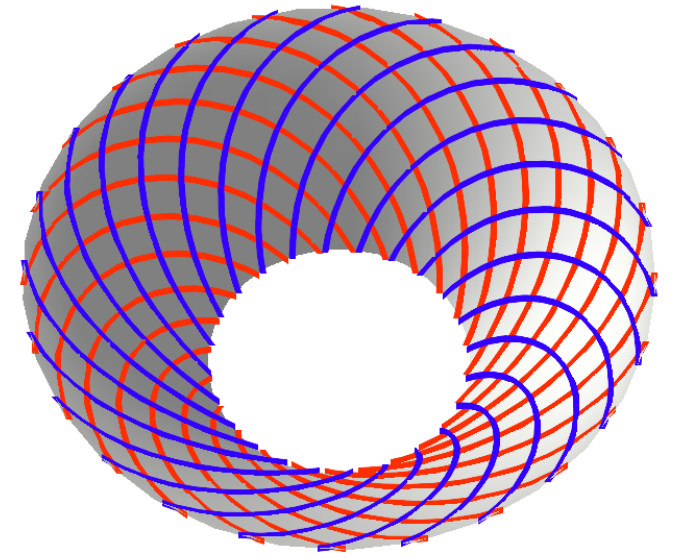}
\caption{The two concordant Chebyshev nets on the middle surface.}
\label{fig:twocoaxial}
\end{figure}
\end{Example}

\appendix

\section{Appendix on relations among the second-order invariants}
\label{app:rel}

As can be inferred from the exposition in Section~\ref{sect:nets},
the geometry of nets in Euclidean space is characterised by the
invariance with respect to rigid motions combined with the
reparameterisations~\eqref{eq:net:transf}.

Consider an isoparametric net $\ve r(x_1,x_2)$.
An $r$th-order scalar differential invariant, $r \ge 1$, of the net is a
scalar expression constructed from the derivatives of
$\ve r$ of order $\le r$ invariant with respect to rigid motions and
transformations~\eqref{eq:net:transf},
i.e., with respect to the $r$-jet prolongation~\cite{A-L-V}
of the vector field
\[
F_i(x_i) \pd{}{x_i} + (\ve Q \cdot \ve r + \ve P) \cdot \pd{}{\ve r},
\]
where $\ve Q$ and $\ve P$ stand for arbitrary rotation and translation
matrices, respectively, while $F_i(x_i)$ are arbitrary functions.
Computing routinely the number $M^{\rm net}_r$ of functionally independent scalar
differential invariants of order $r$,
we obtain the increments $N^{\rm net}_r = M^{\rm net}_r - M^{\rm net}_{r - 1}$
given in Table~\ref{tab:inv}
(so that $M^{\rm net}_r$ is $N^{\rm net}_1 + \cdots +N^{\rm net}_r$).
For comparison, we also give the analogous increments $N^{\rm surf}_r$
for invariants of surfaces.
\begin{table}[ht]\centering
\def\hn{\phantom{0}}\renewcommand{\arraystretch}{1.2}
\begin{tabular}{r|rrrrrrrcr}
$\text{order } r$ & 0 & \hn 1 & \hn 2 & \hn 3 &  4 &  5 & \ldots & $\hn r$ & \ldots
\\\hline
 $N^{\rm net}_r$ & 0 & 1 & 7 & 10 & 13 & 16 & \ldots & $3 r + 1$ & \ldots
\\
 $N^{\rm surf}_r$ & 0 & 0 & 2 & 4 & 5 & 6 & \ldots & $\hn r + 1$ & \ldots
\end{tabular}
\caption{Growth table of the number of invariants of order $r$.}
\label{tab:inv}
\end{table}

As we can see, for surfaces there are just two independent invariants of the second order
that can be used to specify a geometric class of surfaces.
In contrast, as much as eight independent second-order invariants may be
involved in the specification of a geometric class of nets.

The following simple proposition yields another upper bound on the number of
independent invariants.

\begin{Proposition}
There exist no more than four functionally independent scalar invariants
expressible in terms of $\I ij$, $\II ij$.
\end{Proposition}

\begin{proof}
We have six
independent components $\I ij$, $\II ij$ and two independent parameters $f_i$.
\end{proof}

\begin{Proposition}
In the generic case,
the eight independent invariants of order $\le 2$ predicted in Table~\rmref{tab:inv}
can be chosen to be the union of any two of
$\{\omega,\sigma\}$, $\{K,H\}$, $\{\nc1,\nc2\}$, $\{\gtor1,\gtor2\}$ along with any
two of $\{\gc1,\gc2\}$, $\{\pi_1,\pi_2\}$, $\{\iota_1,\iota_2\}$,
$\bigl\{\widehat{X_1} \omega, \widehat{X_2} \omega\bigr\}$.
\end{Proposition}

\begin{proof}
A straightforward proof goes by computation of Jacobi determinants.
\end{proof}

The above results imply the existence of mutual relations.
A number of them can be found in~\cite{WS-I,WS-II,WS-III}, \cite[Chapter~4]{Spi-III},
\cite[Section~93]{VIS}, and later in this section.

Among the known relations we mention the
Beetle identities~\cite[equation~(10)]{Beetle}
\[
\gtor i^2 + \nc i^2 - 2 H \nc i + K = 0
\]
and
\begin{gather*}
{\gtor 1} + {\gtor 2} = (\nc 2 - \nc 1) \cot\omega,
\\
K = {\nc1 \nc2 + \gtor1 \gtor2} + (\nc1 \gtor2 - \nc2 \gtor1) \cot \omega,
\\
2 H = {\nc1 + \nc2} + (\gtor2 - \gtor1) \cot \omega,
\end{gather*}
see~\cite{WS-I,WS-II,WS-III}.
These are polynomial relations homogeneous
with respect to the weight equal to the degree in $\II ij$.
Let us look for similar identities incorporating the Schief curvature.
Invariants rational in $\II ij$ can be routinely
expressed in terms of $\omega$, $\sigma$, $\nc 1$, $\nc 2$
by substituting $\II ii = \nc i \I ii$ and
$\II 12 = \sigma \sqrt{\detI}$, followed by expressing the
first-order coefficients in terms of $\omega$.
In this way, we easily obtain
\begin{gather}\label{eq:gtncsigma}
(-1)^i\,{\gtor i} = \nc i \cot\omega - \sigma,
\end{gather}
as well as the identities
\[
\nc1\nc2 = (K + \sigma^2) \sin^2\omega,
\qquad
\nc1 + \nc2 = 2 (H \sin\omega + \sigma \cos\omega) \sin\omega,
\]
from which one can express the curvatures $\nc1$ and $\nc2$ in terms
of $K$, $H$, $\sigma$, $\omega$;
then also~$\gtor1$ and~$\gtor2$ by \eqref{eq:gtncsigma}.
Conversely, if $\cos\omega \ne 0$, then
system~\eqref{eq:gtncsigma} can be solved for $\sigma$
and $\cot\omega$ as
\[
\sigma = \frac{\gtor1 \nc2 + \gtor2 \nc1}{\nc2 - \nc1},
\qquad
\cot\omega = \frac{\gtor1 + \gtor2}{\nc2 - \nc1}.
\]

Finally, the invariants $\iota_i$, $\pi_i$, $\gc i$, $\widehat{X_i} \omega$ expressible
in terms of $\I ij$, $\widehat{X_i}$ are related by
\begin{gather}\label{eq:pigc}
\pi_i \sin\omega + \widehat{X_i} \omega  = (-1)^i \gc i
\end{gather}
and
\begin{gather}\label{eq:iotapi}
\pi_1 + \pi_2 \cos\omega = \iota_1, \qquad
\pi_1 \cos\omega + \pi_2 = {-\iota_2}.
\end{gather}
All these formulas can be proved by straightforward computation.

Let us also mention some simple vector invariants.
Recall that $\widehat{X_1} \ve r$, $\widehat{X_2} \ve r$ are the unit tangent vectors along
the curves of the net.
The vectors $\widehat{X_1} \widehat{X_2} \ve r$, $\widehat{X_2} \widehat{X_1} \ve r$ are
two different invariant versions of what is often referred to as the {\it twist} in
computational geometry
(\cite[end of Section~7.1]{F-P-1979} or~\cite{B-F-F-H}).
It is easily checked that
\begin{gather}\label{eq:XYvect}
\begin{split}
&\widehat{X_1} \widehat{X_2} \ve r = (\sigma \sin\omega) \ve n
 + \pi_1 \widehat{X_1} \ve r - (\pi_1 \cos\omega) \widehat{X_2} \ve r,
\\
& \widehat{X_2} \widehat{X_1} \ve r = (\sigma \sin\omega) \ve n
 - \pi_2 \widehat{X_1} \ve r + (\pi_2 \cos\omega) \widehat{X_2} \ve r.
 \end{split}
\end{gather}
Then
\[
\bigl[\widehat{X_1}, \widehat{X_2}\bigr] \ve r = \widehat{X_1} \widehat{X_2} \ve r - \widehat{X_2} \widehat{X_1} \ve r
 = (\pi_1 + \pi_2 \cos\omega) \widehat{X_1} \ve r
 - (\pi_1 \cos\omega + \pi_2) \widehat{X_2} \ve r
\]
proves formula~\eqref{eq:iotapi}.
Furthermore, $\widehat{X_1} \ve n$, $\widehat{X_2} \ve n$ are tangent vectors to the surface that
reflect the change of the normal vector to the surface along the curves of the net.
In matrix notation, we have
\begin{gather*}
\begin{pmatrix} \widehat{X_1} \ve r \\ \widehat{X_2} \ve r \end{pmatrix}
\begin{pmatrix} \widehat{X_1} \ve n & \widehat{X_2} \ve n \end{pmatrix}
 = -\begin{pmatrix} \widehat{X_1} \widehat{X_1} \ve r & \widehat{X_1} \widehat{X_2} \ve r \\
    \widehat{X_2} \widehat{X_1} \ve r & \widehat{X_2} \widehat{X_2} \ve r
    \end{pmatrix} \cdot \ve n
 = -\begin{pmatrix} \nc1 & \sigma \\ \sigma & \nc2\end{pmatrix} \sin\omega,
\end{gather*}
which demonstrates a kinship between $\sigma$ and the normal curvatures.

Finally, $\widehat{X_1} \ve n \cdot \widehat{X_2} \ve n$
equals $K \sin \omega \cot \omegaIII$, where $\omegaIII$
is the intersection angle of the spherical image of the net.
Moreover,
\begin{gather}\label{eq:cotphiIII}
\cot\omegaIII
 = \frac{2 H \sigma}{K} - \cot\omega.
\end{gather}

To conclude this section, we review five discrete symmetries of nets
described in Tables~\ref{tab:Ts} and~\ref{tab:Ts:XnFF}.
Their action on the invariants is summarised in Table~\ref{tab:Ts:inv}.

\begin{table}[th]\centering\renewcommand{\arraystretch}{1.2}
\begin{tabular}{r|l}
$T_{-1}$ & Reversion of the protractor, $\omega \longleftrightarrow -\omega$ \\
$T_0$ & Change of sign of all vector and triple products
   (the orientation of Euclidean space) \\
$T_1$ & Change of orientation of curves of the first family \\
$T_2$ & Change of orientation of curves of the second family \\
$T_3$ & Family swap
\end{tabular}
\caption{Five discrete symmetries of nets in Euclidean space.}
\label{tab:Ts}
\end{table}

\begin{table}[th]\centering\renewcommand{\arraystretch}{1.2}
\begin{tabular}{r|rrrrrrrrrr}
     &  $X_1$ &  $X_2$ & $\ve n$ &$\hphantom{-}\I11$ &  $\I12$ &$\hphantom{-}\I22$ &  $\II11$ &  $\II12$ &  $\II22$ \\
     \hline
$T_\1$ &  $X_1$ &  $X_2$ &  $\ve n$ &  $\I11$ & $-\I12$ &  $\I22$ &  $\II11$ &  $\II12$ &  $\II22$ \\
 $T_0$ &  $X_1$ &  $X_2$ & $-\ve n$ &  $\I11$ &  $\I12$ &  $\I22$ & $-\II11$ & $-\II12$ & $-\II22$ \\
 $T_1$ & $-X_1$ &  $X_2$ &  $\ve n$ &  $\I11$ & $-\I12$ &  $\I22$ &  $\II11$ & $-\II12$ &  $\II22$ \\
 $T_2$ &  $X_1$ & $-X_2$ &  $\ve n$ &  $\I11$ & $-\I12$ &  $\I22$ &  $\II11$ & $-\II12$ &  $\II22$ \\
 $T_3$ &  $X_2$ &  $X_1$ &  $\ve n$ &  $\I22$ &  $\I12$ &  $\I11$ &  $\II22$ &  $\II12$ &  $\II11$
\end{tabular}

\caption{The action of discrete symmetries on $X_i$, $\ve n$, $\I ij$ and $\II ij$.}\label{tab:Ts:XnFF}
\end{table}

\begin{table}[!ht]\centering\renewcommand{\arraystretch}{1.2}\setlength{\tabcolsep}{3.5pt}

\begin{tabular}{r|rrrrrrrrrrrrrrr}
     &  $\omega$    &  $H$ &  $\sigma$ &  $\nc2$ &  $\nc1$ &  $\gc1$ &  $\gc2$ &  $\gtor1$ &  $\gtor2$ &  $\pi_1$ &  $\pi_2$ &  $\iota_1$ &  $\iota_2$  & $\omega_{,1}$ &  $\omega_{,2}$  \\
\hline
 $T_\1$& $-\omega$    &  $H$ & $-\sigma$ &  $\nc1$ &  $\nc2$ & $-\gc1$ & $-\gc2$ & $-\gtor1$ & $-\gtor2$ &  $\pi_1$ &   $\pi_2$ & $\iota_1$ &  $\iota_2$ & $-\omega_{,1}$ &  $\omega_{,2}$ \\
 $T_0$ &  $\omega$    & $-H$ & $-\sigma$ & $-\nc1$ & $-\nc2$ &  $\gc1$ &  $\gc2$ & $-\gtor1$ & $-\gtor2$ &  $\pi_1$ &   $\pi_2$ &  $\iota_1$ &  $\iota_2$   & $\omega_{,1}$ &  $\omega_{,2}$ \\
 $T_1$ & $\pi - \omega$ & $ H$ & $-\sigma$ &  $\nc1$ &  $\nc2$ &  $\gc1$ & $-\gc2$ & $-\gtor1$ & $-\gtor2$ &  $\pi_1$ & $-\pi_2$ & $\iota_1$ &  $-\iota_2$  & $\omega_{,1}$ & $-\omega_{,2}$  \\
 $T_2$ & $\pi - \omega$ &  $H$ & $-\sigma$ &  $\nc1$ &  $\nc2$ & $-\gc1$ &  $\gc2$ & $-\gtor1$ & $-\gtor2$ & $-\pi_1$ &  $\pi_2$ &  $-\iota_1$ &  $\iota_2$ & $-\omega_{,1}$ &  $\omega_{,2}$ \\
 $T_3$ &  $\omega$    &  $H$ &  $\sigma$ &  $\nc2$ &  $\nc1$ & $-\gc2$ & $-\gc1$ & $-\gtor2$ & $-\gtor1$ &  $\pi_2$ &  $\pi_1$ &   $-\iota_2$ &  $-\iota_1$   & $\omega_{,2}$ &  $\omega_{,1}$
\end{tabular}
\caption{The action of discrete symmetries on the invariants.}
\label{tab:Ts:inv}
\end{table}

The action on $\omegaIII$ is the same as on $\omega$.
Needless to say, all the identities among invariants we have listed in this section
are invariant under transformations $T_{-1}, \dots, T_3$.

\section*{Conclusions and perspectives}

After reviewing nets and their second-order invariants,
we  introduced integrable classes of nets in analogy with
integrable classes of surfaces.
Then, starting from an earlier result~\cite{K-M}, we established equivalence of
concordant Chebyshev nets and pairs of pseudospherical surfaces.
The integrability of concordant Chebyshev nets, which we first observed
in \cite{K-M}, is hereby explicitly related to the integrability of pseudospherical
surfaces.
Presented examples are the concordant Chebyshev nets on the middle surface of
two pseudospheres and on the middle surface of the pseudosphere and another coaxial
axisymmetric pseudospherical surface.

In the outlook, we identify the following tasks:
\begin{itemize}\itemsep=0pt
\item[--] Explore the full ``parameter space'' in~Figure~\ref{fig:world}.
\item[--] Explore the other integrable Chebyshev nets from paper~\cite{K-M}.
\item[--] Employ the ZCRs found in paper~\cite{K-M} to obtain recursion operators,
solutions, etc.
\item[--] Use the methods of papers~\cite{B-M-2010,K-M} to search for new integrable
classes of nets.
\item[--] Explore higher-dimensional analogues.
\end{itemize}

\subsection*{Acknowledgements}

This research received support from M\v{S}MT under RVO 47813059.
The author is grateful to Evgeny Ferapontov and Jan Cie\'sli\'nski for introduction
into integrable surfaces and thought-provoking discussions that inspired this
particular research.


\begin{thebibliography}{99}
\footnotesize\itemsep=0pt

\bibitem{SIA-2021}
Agafonov S.I., Quadratic integrals of geodesic flow, webs, and integrable
  billiards, \href{https://doi.org/10.1016/j.geomphys.2020.104041}{\textit{J.~Geom. Phys.}} \textbf{161} (2021), 104041, 7~pages,
  \href{https://arxiv.org/abs/2004.12374}{arXiv:2004.12374}.

\bibitem{A-L-V}
Alekseevskij D.V., Vinogradov A.M., Lychagin V.V., Basic ideas and concepts of
  differential geometry, in Geometry~{I}, Editor R.V.~Gamkrelidze, \textit{Encyclopaedia Math. Sci.},
  Vol.~28, Springer, Berlin, 1991, 1--264.

\bibitem{AA}
Aoust A., Analyse infinit\'esimale des courbes trac\'ees sur une surface
  quelconque, Gauthier-Villars, Paris, 1869.

\bibitem{B-M-2009}
Baran H., Marvan M., On integrability of {W}eingarten surfaces: a~forgotten
  class, \href{https://doi.org/10.1088/1751-8113/42/40/404007}{\textit{J.~Phys.~A}} \textbf{42} (2009), 404007, 16~pages,
  \href{https://arxiv.org/abs/1002.0989}{arXiv:1002.0989}.

\bibitem{B-M-2010}
Baran H., Marvan M., Classification of integrable {W}eingarten surfaces
  possessing an {$\mathfrak{sl}(2)$}-valued zero curvature representation,
  \href{https://doi.org/10.1088/0951-7715/23/10/013}{\textit{Nonlinearity}} \textbf{23} (2010), 2577--2597, \href{https://arxiv.org/abs/1002.0992}{arXiv:1002.0992}.

\bibitem{B-F-F-H}
Barnhill R.E., Farin G., Fayard L., Hagen H., Twists, curvatures and surface
  interrogation, \href{https://doi.org/10.1016/0010-4485(88)90116-9}{\textit{Comput. Aided Des.}} \textbf{20} (1988), 341--344,
  345--346.

\bibitem{Beetle}
Beetle R.D., A~formula in the theory of surfaces, \href{https://doi.org/10.2307/1967816}{\textit{Ann. of Math.}}
  \textbf{15} (1913--1914), 179--183.

\bibitem{LB-I}
Bianchi L., Lezioni di Geometria Differenziale, Vol.~1, E.~Spoerri, Pisa, 1902.

\bibitem{LB-II}
Bianchi L., Lezioni di Geometria Differenziale, Vol.~2, E.~Spoerri, Pisa, 1903.

\bibitem{LB-1922}
Bianchi L., Le reti di {T}chebychef sulle superficie ed il parallelismo nel
  senso di {L}evi-{C}ivita, \textit{Boll. Unione Mat. Ital.} \textbf{1} (1922),
  1--6.

\bibitem{Bob}
Bobenko A.I., Surfaces in terms of~{$2$} by~{$2$} matrices. {O}ld and new
  integrable cases, in Harmonic Maps and Integrable Systems, Editors A.P.~Fordy, J.C.~Wood, \textit{Aspects
  Math.}, Vol. E23, \href{https://doi.org/10.1007/978-3-663-14092-4_5}{Friedr. Vieweg}, Braunschweig, 1994, 83--127.

\bibitem{AB-UP}
Bobenko A.I., Pinkall U., Discretization of surfaces and integrable systems, in
  Discrete Integrable Geometry and Physics ({V}ienna, 1996),  Editors A.I.~Bobenko, R.~Seiler, \textit{Oxford
  Lecture Ser. Math. Appl.}, Vol.~16, Oxford University Press, New York, 1999,
  3--58.

\bibitem{B-P-R}
Bobenko A.I., Pottmann H., R\"orig T., Multi-nets. {C}lassification of discrete
  and smooth surfaces with characteristic properties on arbitrary parameter
  rectangles, \href{https://doi.org/10.1007/s00454-019-00101-1}{\textit{Discrete Comput. Geom.}} \textbf{63} (2020), 624--655,
  \href{https://arxiv.org/abs/1802.05063}{arXiv:1802.05063}.

\bibitem{B-S}
Bobenko A.I., Suris Yu.B., Discrete differential geometry. Integrable structure,
  \textit{Grad. Stud. Math.}, Vol.~98, \href{https://doi.org/10.1090/gsm/098}{American Mathematical Society},
  Providence, RI, 2008.

\bibitem{EB-1925}
Bortolotti E., Su alcune questioni di geometria delle superficie, \textit{Boll.
  Unione Mat. Ital} \textbf{4} (1925), 162--166.

\bibitem{DB-2017}
Brander D., Pseudospherical surfaces with singularities, \href{https://doi.org/10.1007/s10231-016-0601-8}{\textit{Ann. Mat. Pura
  Appl.}} \textbf{196} (2017), 905--928, \href{https://arxiv.org/abs/1502.04876}{arXiv:1502.04876}.

\bibitem{B-T-1995}
Bruce J.W., Tari F., On binary differential equations, \href{https://doi.org/10.1088/0951-7715/8/2/008}{\textit{Nonlinearity}}
  \textbf{8} (1995), 255--271.

\bibitem{DCF-LAOS}
Catalano~Ferraioli D., de~Oliveira~Silva L.A., Nontrivial 1-parameter families
  of zero-curvature representations obtained via symmetry actions,
  \href{https://doi.org/10.1016/j.geomphys.2015.04.001}{\textit{J.~Geom. Phys.}} \textbf{94} (2015), 185--198.

\bibitem{PLC}
Chebyshev P.L., On the cutting of garments, \textit{Russian Math. Surveys}
  \textbf{1} (1946), no.~2, 38--42.

\bibitem{C-T-1981}
Chern S.-S., Tenenblat K., Foliations on a~surface of constant curvature and the
  modified {K}orteweg--de {V}ries equations, \href{https://doi.org/10.4310/jdg/1214436216}{\textit{J.~Differential Geometry}}
  \textbf{16} (1981), 347--349.

\bibitem{JC-1993}
Cie\'sli\'nski J., Nonlocal symmetries and a~working algorithm to isolate
  integrable geometries, \href{https://doi.org/10.1088/0305-4470/26/5/017}{\textit{J.~Phys.~A}} \textbf{26} (1993), L267--L271.

\bibitem{C-G-S}
Cie\'sli\'nski J., Goldstein P., Sym A., Isothermic surfaces in
  {$\mathbf{E}^3$} as soliton surfaces, \href{https://doi.org/10.1016/0375-9601(95)00504-V}{\textit{Phys. Lett.~A}} \textbf{205}
  (1995), 37--43, \href{https://arxiv.org/abs/solv-int/9502004}{arXiv:solv-int/9502004}.

\bibitem{GD-I}
Darboux G., Le\c{c}ons sur la th\'eorie g\'en\'erale des surfaces, Premi\`ere
  Partie, Gauthier-Villars, Paris, 1887.

\bibitem{GD-II}
Darboux G., Le\c{c}ons sur la th\'eorie g\'en\'erale des surfaces, Deuxi\`eme
  Partie, Gauthier-Villars, Paris, 1889.

\bibitem{GD-III}
Darboux G., Le\c{c}ons sur la th\'eorie g\'en\'erale des surfaces, Troisi\`eme
  Partie, Gauthier-Villars, Paris, 1894.

\bibitem{AMD}
D\'ecaillot A.-M., G\'eom\'etrie des tissus. {M}osa\"{\i}ques. \'Echiquiers.
  {M}ath\'ematiques curieuses et utiles, \textit{Rev. Histoire Math.}
  \textbf{8} (2002), 145--206.

\bibitem{UD-1870}
Dini U., Sopra alcune formole generali della teoria delle superficie, e loro
  applicazioni, \href{https://doi.org/10.1007/BF02420031}{\textit{Ann. Mat. Pura Appl.}} \textbf{4} (1870), 175--206.

\bibitem{YaSD}
Dubnov Ya.S., Semitensors of a~two-dimensional net, \textit{Izv. Vyssh. Uchebn.
  Zaved. Mat.}  (1958), no.~3, 74--83.

\bibitem{BAD-1981}
Dubrovin B.A., Theta-functions and nonlinear equations, \href{https://doi.org/10.1070/RM1981v036n02ABEH002596}{\textit{Russian Math.
  Surveys}} \textbf{36} (1981), no.~2, 11--92.

\bibitem{LPE-1904}
Eisenhart L.P., Three particular systems of lines on a~surface, \href{https://doi.org/10.2307/1986272}{\textit{Trans.
  Amer. Math. Soc.}} \textbf{5} (1904), 421--437.

\bibitem{LPE-1909}
Eisenhart L.P., A~treatise on the differential geometry of curves and surfaces,
  Ginn, Boston, 1909.

\bibitem{F-P-1979}
Faux I.D., Pratt M.J., Computational geometry for design and manufacture, Math.
  Appl., Ellis Horwood, Chichester, 1979.

\bibitem{F-T-2007}
Fuchs D., Tabachnikov S., Mathematical omnibus. Thirty lectures on classic
  mathematics, \href{https://doi.org/10.1090/mbk/046}{American Mathematical Society}, Providence, RI, 2007.

\bibitem{BG}
Gambier B., Surfaces de {V}oss-{G}uichard, \href{https://doi.org/10.24033/asens.813}{\textit{Ann. Sci. \'Ecole Norm.
  Sup.}} \textbf{48} (1931), 359--396.

\bibitem{Ghy}
Ghys \'E., Sur la coupe des v\^etements: variation autour d'un th\`eme de
  {T}chebychev, \href{https://doi.org/10.4171/LEM/57-1-8}{\textit{Enseign. Math.}} \textbf{57} (2011), 165--208.

\bibitem{WCG-1922}
Graustein W.C., Parallel maps of surfaces, \href{https://doi.org/10.2307/1988813}{\textit{Trans. Amer. Math. Soc.}}
  \textbf{23} (1922), 298--332.

\bibitem{WCG-1932}
Graustein W.C., Parallelism and equidistance in classical differential
  geometry, \href{https://doi.org/10.2307/1989367}{\textit{Trans. Amer. Math. Soc.}} \textbf{34} (1932), 557--593.

\bibitem{GMG}
Green G.M., Nets of space curves, \href{https://doi.org/10.2307/1988903}{\textit{Trans. Amer. Math. Soc.}} \textbf{21}
  (1920), 207--236.

\bibitem{G-T}
G\"urses M., Tek S., Integrable curves and surfaces,
in Proceedings of the 17th International Conference
``Geometry, Integrability and Quantization'' (Sts. Constantine and Elena, June 5--10, 2015), Editors I.M.~Mladenov,
G.~Meng, A.~Yoshioka, \textit{Geom. Integrability Quantization}, Vol.~17,
  \href{https://doi.org/10.7546/giq-17-2016-13-71}{Bulgarian Academy of Sciences}, Sofia, 2016, 13--71.

\bibitem{HH-1972}
Hasimoto H., A~soliton on a~vortex filament, \href{https://doi.org/10.1017/S0022112072002307}{\textit{J.~Fluid Mech.}}
  \textbf{51} (1972), 477--485.

\bibitem{H-O-O}
Howell P.D., Ockendon H., Ockendon J.R., Draping woven sheets,
  \href{https://doi.org/10.1017/S144618112000019X}{\textit{ANZIAM~J.}} \textbf{62} (2020), 355--385.

\bibitem{VFK-II}
Kagan V.F., Fundamentals of the theory of surfaces in tensorial
  representation~{II}, Gostechizdat, Moskva, 1948.

\bibitem{K-T-2021}
Khesin B., Tabachnikov S., Polar bear or penguin? {M}usings on earth
  cartography and {C}hebyshev nets, \href{https://doi.org/10.1007/s00283-020-10013-1}{\textit{Math. Intelligencer}} \textbf{43}
  (2021), 20--24.

\bibitem{K-M-T}
Kilian M., M\"uller C., Tervooren J., Smooth and discrete cone-nets,
  \href{https://doi.org/10.1007/s00025-023-01884-9}{\textit{Results Math.}} \textbf{78} (2023), 110, 40~pages.

\bibitem{RK-1979}
Koch R., Parallelogrammnetze, \href{https://doi.org/10.1007/BF01300243}{\textit{Monatsh. Math.}} \textbf{86} (1979),
  265--284.

\bibitem{RK-1982}
Koch R., Diagonale {T}schebyscheff--{N}etze, \href{https://doi.org/10.1007/BF02941865}{\textit{Abh. Math. Sem. Univ.
  Hamburg}} \textbf{52} (1982), 43--66.

\bibitem{BGK-1993}
Konopel'chenko B.G., Nets in {$\mathbb{R}^3$}, their integrable evolutions and
  the {DS} hierarchy, \href{https://doi.org/10.1016/0375-9601(93)91162-X}{\textit{Phys. Lett.~A}} \textbf{183} (1993), 153--159.

\bibitem{K-M}
Krasil'shchik J., Marvan M., Coverings and integrability of the
  {G}auss--{M}ainardi--{C}odazzi equations, \href{https://doi.org/10.1023/A:1006121716159}{\textit{Acta Appl. Math.}}
  \textbf{56} (1999), 217--230, \href{https://arxiv.org/abs/solv-int/9812010}{arXiv:solv-int/9812010}.

\bibitem{ENK}
Kuznetsov E.N., Underconstrained structural systems, Mech. Engrg. Ser.,
  \href{https://doi.org/10.1007/978-1-4612-3176-9}{Springer}, New York, 1991.

\bibitem{GLL-1977}
Lamb Jr. G.L., Solitons on moving space curves, \href{https://doi.org/10.1063/1.523453}{\textit{J.~Math. Phys.}}
  \textbf{18} (1977), 1654--1661.

\bibitem{L-S-T}
Levi D., Sym A., Tu G.-Z., A~working algorithm to isolate integrable surfaces in~$E^3$, {P}reprint DF INFN Roma 761, 1990.

\bibitem{IL-2015}
Liddell I., Frei {O}tto and the development of gridshells, \href{https://doi.org/10.1016/j.csse.2015.08.001}{\textit{Case Stud.
  Struct. Eng.}} \textbf{4} (2015), 39--49.

\bibitem{SL-1878}
Lie S., Beitr\"age zur {T}heorie der {M}inimalfl\"achen.~{I}. {P}rojectivische {U}ntersuchungen \"uber algebraische {M}inimalfl\"achen, \href{https://doi.org/10.1007/BF01677141}{\textit{Math. Ann.}} \textbf{14} (1878), 331--416.

\bibitem{RvL}
von Lilienthal R., Die auf einer Fl\"ache gezogenen Kurven, in Enzyklop\"adie
  der Mathematischen Wissenschaften mit Einschlu{\ss} ihrer Anwendungen,
  Teubner, Leipzig, 1902, 105--183.

\bibitem{RL-RC}
Lin R., Conte R., On a~surface isolated by {G}ambier, \href{https://doi.org/10.1080/14029251.2018.1503393}{\textit{J.~Nonlinear
  Math. Phys.}} \textbf{25} (2018), 509--514, \href{https://arxiv.org/abs/1805.10450}{arXiv:1805.10450}.

\bibitem{DL-2023}
Liu D., Pellis D., Chiang Y.-C., Rist F., Wallner J., Pottmann H., Deployable
  strip structures, \href{https://doi.org/10.1145/3592393}{\textit{ACM Trans. Graph.}} \textbf{42} (2023), 103,
  16~pages.

\bibitem{M-T-1956}
Mack C., Taylor H.M., The fitting of woven cloth to surfaces, \href{https://doi.org/10.1080/19447027.1956.10750433}{\textit{J.~Text.
  Inst. Trans.}} \textbf{47} (1956), T477--T488.

\bibitem{GM-1961}
Margulies G., Peterson's theorem on surfaces corresponding by parallelism.~{I},
  \href{https://doi.org/10.2307/2034248}{\textit{Proc. Amer. Math. Soc.}} \textbf{12} (1961), 577--587.

\bibitem{MM-rzcr}
Marvan M., Reducibility of zero curvature representations with application to
  recursion operators, \href{https://doi.org/10.1023/B:ACAP.0000035588.67805.0b}{\textit{Acta Appl. Math.}} \textbf{83} (2004), 39--68,
  \href{https://arxiv.org/abs/nlin.SI/0306006}{arXiv:nlin.SI/0306006}.

\bibitem{MM-spp}
Marvan M., On the spectral parameter problem, \href{https://doi.org/10.1007/s10440-009-9450-4}{\textit{Acta Appl. Math.}}
  \textbf{109} (2010), 239--255, \href{https://arxiv.org/abs/0804.2031}{arXiv:0804.2031}.

\bibitem{YM-2017}
Masson Y., Existence and construction of Chebyshev nets with singularities and
  application to gridshells, Ph.D. thesis, {U}niversit\'e Paris Est, 2017,
  available at \url{https://theses.hal.science/tel-01676984v1}.

\bibitem{M-M-2017}
Masson Y., Monasse L., Existence of global {C}hebyshev nets on surfaces of
  absolute {G}aussian curvature less than~{$2\pi$}, \href{https://doi.org/10.1007/s00022-016-0319-1}{\textit{J.~Geom.}}
  \textbf{108} (2017), 25--32.

\bibitem{VBM-2008}
Matveev V.B., 30 years of finite-gap integration theory, \href{https://doi.org/10.1098/rsta.2007.2055}{\textit{Philos.
  Trans.~R. Soc. Lond. Ser.~A Math. Phys. Eng. Sci.}} \textbf{366} (2008),
  837--875.

\bibitem{M-D-T-F-B}
Montagne N., Douthe C., Tellier X., Fivet C., Baverel O., Voss surfaces: {A}
  design space for geodesic gridshells, in Inspiring the Next Generation (Proc.
  IASS Annual Symp. 2020 and 7th Int. Conf. Spatial Structures), Editors S.A.~Behnejad, G.A.R.~Parke and O.A.~Samavati, \href{https://doi.org/10.15126/900337}{Spatial Structures Research Centre University Surrey}, UK,
  2021, 3473--3483.

\bibitem{M-B-2008}
Mukherjee R., Balakrishnan R., Moving curves of the sine-{G}ordon equation: new
  links, \href{https://doi.org/10.1016/j.physleta.2008.08.070}{\textit{Phys. Lett.~A}} \textbf{372} (2008), 6347--6362.

\bibitem{APN}
Norden A.P., Theory of surfaces, GITTL, Moscow, 1956.

\bibitem{KMP-1866}
Peterson K.M., Sur les relations et les affinit\'es entre les surfaces courbes,
  \textit{Ann. Fac. Sci. Univ. Toulouse} \textbf{7} (1905), 5--43, translated from Russian: \textit{Mat. Sb.}
  \textbf{1} (1866), 391--438.

\bibitem{P-E-V-W}
Pottmann H., Eigensatz M., Vaxman A., Wallner J., Architectural geometry,
  \href{https://doi.org/10.1016/j.cag.2014.11.002}{\textit{Comput. Graph.}} \textbf{47} (2015), 145--164.

\bibitem{R-S}
Rogers C., Schief W.K., B\"acklund and {D}arboux transformations. Geometry and
  modern applications in soliton theory, \textit{Cambridge Texts Appl. Math.}, \href{https://doi.org/10.1017/CBO9780511606359}{Cambridge
  University Press}, Cambridge, 2002.

\bibitem{SF-C-BC-V}
Sageman-Furnas A.O., Chern A., Ben-Chen M., Vaxman A., {C}hebyshev nets from
  commuting {PolyVector} fields, \href{https://doi.org/10.1145/3355089.3356564}{\textit{ACM Trans. Graph.}} \textbf{38} (2019),
  172, 16~pages.

\bibitem{GS}
Sannia G., Geometria differenziale dei reticolati piani invariante per un
  gruppo di~collineazioni, \href{https://doi.org/10.1007/BF03014704}{\textit{Rend. Circ. Mat. Palermo}} \textbf{48}
  (1924), 289--307.

\bibitem{RS-1933}
Sauer R., Weckelige {K}urvennetze bei einer infinitesimalen
  {F}l\"achenverbiegung, \href{https://doi.org/10.1007/BF01452858}{\textit{Math. Ann.}} \textbf{108} (1933), 673--693.

\bibitem{RS-1970}
Sauer R., Differenzengeometrie, \href{https://doi.org/10.1007/978-3-642-86411-7}{Springer}, Berlin, 1970.

\bibitem{WS-I}
Scherrer W., Die {G}rundgleichungen der {F}l\"achentheorie.~{I},
  \href{https://doi.org/10.1007/BF02564278}{\textit{Comment. Math. Helv.}} \textbf{29} (1955), 180--198.

\bibitem{WS-II}
Scherrer W., Die {G}rundgleichungen der {F}l\"achentheorie.~{II},
  \href{https://doi.org/10.1007/BF02564571}{\textit{Comment. Math. Helv.}} \textbf{32} (1957), 73--84.

\bibitem{WS-III}
Scherrer W., Die {G}rundgleichungen der {F}l\"achentheorie.~{III},
  \href{https://doi.org/10.1007/BF02566971}{\textit{Comment. Math. Helv.}} \textbf{37} (1962), 177--197.

\bibitem{WKS-2003}
Schief W.K., On the integrability of {B}ertrand curves and {R}azzaboni
  surfaces, \href{https://doi.org/10.1016/S0393-0440(02)00130-4}{\textit{J.~Geom. Phys.}} \textbf{45} (2003), 130--150.

\bibitem{WKS-2007}
Schief W.K., Discrete {C}hebyshev nets and a~universal permutability theorem,
  \href{https://doi.org/10.1088/1751-8113/40/18/007}{\textit{J.~Phys.~A}} \textbf{40} (2007), 4775--4801.

\bibitem{S-R-1999}
Schief W.K., Rogers C., Binormal motion of curves of constant curvature and
  torsion. {G}eneration of soliton surfaces, \href{https://doi.org/10.1098/rspa.1999.0445}{\textit{R.~Soc. Lond. Proc. Ser.~A
  Math. Phys. Eng. Sci.}} \textbf{455} (1999), 3163--3188.

\bibitem{MS-1904}
Servant M., Sur l'habillage des surfaces, \textit{C.~R.~Acad. Sci.}
  \textbf{137} (1903), 112--115.

\bibitem{HJS}
Shin H.J., Lund--{R}egge vortex strings in terms of {W}eierstrass elliptic
  functions, \href{https://doi.org/10.1016/S0550-3213(01)00645-9}{\textit{Nuclear Phys.~B}} \textbf{624} (2002), 431--451.

\bibitem{VIS}
Shulikovsky V.I., Classical differential geometry, GIFML, Moscow, 1963.

\bibitem{Spi-III}
Spivak M., A~comprehensive introduction to differential geometry. {V}ol.~{III},
  Publish or Perish, Inc., Boston, MA, 1975.

\bibitem{Sym}
Sym A., Soliton surfaces and their applications (soliton geometry from spectral
  problems), in Geometric Aspects of the {E}instein Equations and Integrable
  Systems ({S}cheveningen, 1984), \textit{Lecture Notes in Phys.}, Vol. 239,
  \href{https://doi.org/10.1007/3-540-16039-6_6}{Springer}, Berlin, 1985, 154--231.

\bibitem{XT-2022}
Tellier X., Bundling elastic gridshells with alignable nets. Part~{I}:
  {A}nalytical approach, \href{https://doi.org/10.1016/j.autcon.2022.104291}{\textit{Autom. Constr.}} \textbf{141} (2022), 104291,
  36~pages.

\bibitem{CT-1988}
Tian C., Foliation on a~surface of constant curvature and some nonlinear
  evolution equations, \textit{Chinese Ann. Math. Ser.~B} \textbf{9} (1988),
  118--122.

\bibitem{MT-2017}
Toda M., On a~duality property of isothermic surfaces, \href{https://doi.org/10.17654/GT020010085}{\textit{JP J.~Geom.
  Topol.}} \textbf{20} (2017), 85--90.

\bibitem{V-B-2010}
Vashpanova T.Yu., Bezkorovaynaya L.L., LGT-network of a~surface and its
  properties, \textit{Visn. Kyiv. Univ. im. Tarasa Shevchenka, Ser. Fiz.-Mat.
  Nauk} \textbf{2010} (2010), no.~2, 7--11.

\bibitem{Vax-2016}
Vaxman A., Campen M., Diamanti O., Panozzo D., Bommes D., Hildebrandt K.,
  Ben-Chen M., Directional field synthesis, design, and processing,
  \href{https://doi.org/10.1111/cgf.12864}{\textit{Comput. Graph. Forum}} \textbf{35} (2016), 545--572.

\bibitem{AV-1869}
Voss A., Ueber ein neues {P}rincip der {A}bbildung krummer {O}berfl\"achen,
  \href{https://doi.org/10.1007/BF01447291}{\textit{Math. Ann.}} \textbf{19} (1881), 1--26.

\bibitem{W-P-2022}
Wang H., Pottmann H., Characteristic parameterizations of surfaces with
  a~constant ratio of principal curvatures, \href{https://doi.org/10.1016/j.cagd.2022.102074}{\textit{Comput. Aided Geom. Design}}
  \textbf{93} (2022), 102074, 23~pages.

\bibitem{CEW}
Weatherburn C.E., On {L}evi-{C}ivita's theory of parallelism, \href{https://doi.org/10.1090/S0002-9904-1928-04629-3}{\textit{Bull.
  Amer. Math. Soc.}} \textbf{34} (1928), 585--590.

\bibitem{KHW-1940}
Weise K.H., Invariante {C}harakterisierung von {K}urvennetzen,
  \href{https://doi.org/10.1007/BF01181462}{\textit{Math.~Z.}} \textbf{46} (1940), 665--691.

\bibitem{JKW}
Whittemore J.K., Total geodesic curvature and geodesic torsion, \href{https://doi.org/10.1090/S0002-9904-1923-03659-8}{\textit{Bull.
  Amer. Math. Soc.}} \textbf{29} (1923), 51--54.

\bibitem{Wolf}
Wolf T., A~comparison of four approaches to the calculation of conservation
  laws, \href{https://doi.org/10.1017/S0956792501004715}{\textit{European J.~Appl. Math.}} \textbf{13} (2002), 129--152.

\bibitem{Z-S-1974}
Zakharov V.E., Shabat A.B., Integration of nonlinear equations of
  mathematical physics by the method of inverse scattering.~{II},
  \href{https://doi.org/10.1007/BF01077483}{\textit{Funct. Anal. Appl.}} \textbf{13} (1979), 166--174.

\end{thebibliography}

\pdfbookmark[1]{References}{ref}
\LastPageEnding

\end{document}